\documentclass[a4paper]{article}

\usepackage{amsthm, amsmath, amsfonts, amssymb, accents}
\usepackage{graphicx}
\usepackage{srcltx}

\usepackage{dsfont}

\theoremstyle{plain}
\newtheorem{theorem}{Theorem}[section]

\newtheorem{lemma}[theorem]{Lemma}

\theoremstyle{remark}

\numberwithin{equation}{section}

\newcommand{\tht}{\theta}
\newcommand{\Om}{\Omega}
\newcommand{\om}{\omega}
\newcommand{\e}{\varepsilon}
\newcommand{\g}{\gamma}

\renewcommand{\d}{\delta}
\renewcommand{\b}{\beta}
\renewcommand{\a}{\alpha}
\newcommand{\p}{\partial}
\newcommand{\D}{\Delta}
 \newcommand{\E}{\mbox{\rm e}}
 
\newcommand{\vp}{\varphi}
\newcommand{\vk}{\varkappa}

\newcommand{\z}{\zeta}

\newcommand{\Hinf}{W_\infty}
\newcommand{\Ho}{\mathring{W}_2}

\newcommand{\di}{\,d}
\newcommand{\iu}{\mathrm{i}}

\newcommand{\Ups}{\Upsilon}

\newcommand{\la}{\langle}
\newcommand{\ra}{\rangle}

\newcommand{\ope}{\mathcal{H}^\e}
\newcommand{\foe}{\mathfrak{h}^\e}

\newcommand{\opt}{\mathcal{H}^\tht}

\def\opo#1{\mathcal{H}^0_{\mathrm{#1}}}
\def\foo#1{\mathfrak{h}^0_{\mathrm{#1}}}

\def\topo#1{\widetilde{\mathcal{H}}^0_{\mathrm{#1}}}
\def\tfoo#1{\widetilde{\mathfrak{h}}^0_{\mathrm{#1}}}
\def\tu{\widetilde{u}^0}
\def\tv{\widetilde{v}^\e}

\def\fae{\mathfrak{a}^\e}
\def\fao{\mathfrak{a}}

\def\bl{\mathrm{in}}

\def\ex{\mathrm{ex}}
\def\wt#1{\widetilde{#1}}
\def\wh#1{\widehat{#1}}

\DeclareMathOperator{\spec}{\sigma}

\DeclareMathOperator{\supp}{supp}
\DeclareMathOperator{\Dom}{\mathfrak{D}}

\DeclareMathOperator{\sgn}{sgn}

\DeclareMathOperator{\RE}{Re}

\newcounter{assumption}
\begin{document}

\allowdisplaybreaks

\title{\textbf{Homogenization and norm resolvent convergence for elliptic operators in a strip perforated along a curve}}
\author{Denis Borisov\,$^a$, Giuseppe Cardone$^b$, Tiziana
Durante$^c$}
\date{\small
\begin{center}
\begin{quote}
\begin{enumerate}
{\it
\item[$a)$] Institute of Mathematics with computer center, Ufa Scientific Center, Russian Academy of Sciences, Chernyshevsky str. 112, Ufa, 450008,
 \&
Bashkir State Pedagogical University, October St.~3a,
Ufa, 450000,  Russian Federation; \texttt{borisovdi@yandex.ru}
\item[$b)$]
University of Sannio, Department of Engineering, Corso
Garibaldi, 107, 82100 Benevento, Italy;
\texttt{giuseppe.cardone@unisannio.it}
\item[$c)$]
University of Salerno, Department of Information and Electrical Engineering and Applied Mathematics, Via Ponte Don Melillo, 1, 84084, Fisciano (SA), Italy;  \texttt{tdurante@unisa.it}}
\end{enumerate}
\end{quote}
\end{center}}

\maketitle

\begin{abstract}
We consider an infinite planar straight  strip perforated by small holes along a curve. In such domain, we consider a general second order elliptic operator subject to classical boundary conditions on the holes. Assuming that the perforation is non-periodic and satisfies rather weak assumptions, we describe all possible homogenized problems. Our main result is the norm resolvent convergence of the perturbed operator to a homogenized one in various operator norms and the estimates for the rate of convergence. On the basis of the norm resolvent  convergence, we prove the convergence of the spectrum.
\end{abstract}

\begin{quote}
\\

\end{quote}

\section{Introduction}

Problems in perforated domains is one of the classical objects in the modern homogenization theory. It is impossible to cite all the works in this field and we mention only the books \cite{ZKO}, \cite{MCh}, \cite{MCh2}, \cite{OSHY}, \cite{SP}, and some recent papers \cite{art6}, \cite{CNP}--
\cite{CDG},
\cite{art9}
--\cite{art3}, \cite{art1}--
\cite{art5},  see also the references therein. One of the considered models is a perforation along a curve or a manifold. Let us describe briefly a typical formulation of boundary value problems in domains with such perforations.

Given a bounded or an unbounded domain, we choose a manifold of codimension one in this domain. Along this manifold, a perforation is made by small holes. The distances between the holes are assumed to be small. In this perforated domain an elliptic boundary value problem is considered with one of the classical boundary condition on the boundaries of the holes. The main aim is to study the behavior of the solutions to these problems as the sizes of the holes and the distances between them tend to zero.

The first question addressed in studying the above problems was the description of the homogenized problems whose solution are the limits (in some sense) for the solutions of the original problems. This issue was studied in
\cite{art6}, \cite{art9}, \cite{art7},   \cite{art2}, \cite{art3}, \cite{art1}--
\cite{art5}, \cite{MCh2},   see also the references therein. In all these works except \cite{MCh2} the operator was either the Laplacian or the Laplacian plus a constant. In \cite{MCh2} a general elliptic operator was considered. In this book two problems were studied. The first of them is the Dirichlet problem for a general elliptic equation of even order. The second problem is the Neumann problem for a second order elliptic equation. The geometry of the holes and their distribution were arbitrary.
In \cite{art6} there was considered the Poisson equation in a bounded two-dimensional domain with a periodic perforation along the boundary. In \cite{art2} a variational inequality for the Laplacian was considered in an arbitrary domain. The perforation was periodic along a given manifold and it was assumed that the linear size of the holes is much smaller than the distances between the holes. At the boundary of the holes a nonlinear Robin type boundary condition was imposed. In \cite{art3}, \cite{art4} a similar boundary value problems for the Poisson equation (not for a variational inequality) were studied.  Papers  \cite{art9}, \cite{art7}, \cite{art1}, \cite{art5} were devoted to a non-periodic perforation. In \cite{art1}, \cite{art5} the Poisson equation in a bounded multi-dimensional domain was studied. The sizes of the holes and the  distances between them were of different smallness order.  The boundaries of the holes were subject to one of the classical boundary conditions. In \cite{art7} the Dirichlet condition on the boundary of the holes was imposed.
There were considered eigenvalue equations or equations for the resolvent of the Laplacian. The perforation was along the boundary. Similar model was treated also in \cite{art9}, but here the holes were placed randomly, namely, by means of an ergodic dynamical system. The main result of \cite{art6}, \cite{art9}, \cite{art7},   \cite{art2}, \cite{art3}, \cite{art1}--
\cite{art5}, \cite{MCh2} was the classification of the homogenized problems depending on the geometry of the holes, their distribution and the condition on the boundary of the holes. The convergence of the perturbed solutions to the homogenized ones was proven. This convergence was established in a weak or a strong sense. Namely, a typical result stated that a perturbed solution converges to a homogenized one weakly or strongly in $W_2^1$ or strongly in $L_2$ for each fixed right hand side. For some models the estimates for the rate of convergence were established. Being reformulated in terms of the resolvents, the above results say that the resolvents of the original operators in perforated domains converge to the resolvent of the homogenized operators and the convergence is valid in a weak or strong resolvent sense.

In \cite{27}--
\cite{25a}, there were studied the eigenvalue problems for the Laplacian in a domain perforated along a part of the boundary. The sizes of the holes and the distances between them were of the same smallness order. The condition on the boundaries of the holes was the Dirichlet one. The main result of these papers is the asymptotic expansions for the eigenvalues and the eigenfunctions.

Approximately a decade ago a new direction in the homogenization theory was initiated. It was found for the operators with fast periodically oscillating coefficients that their resolvents converge to the resolvents of the homogenized operators in the norm resolvent sense. This was a much stronger result in comparison with known classical results stating just a weak or a strong resolvent convergence. The results on norm resolvent convergence were obtained by M.Sh.~Birman, T.A.~Suslina \cite{Bi03}, \cite{BS2}, \cite{BS5}, \cite{Su2}, \cite{Su1}, \cite{SuKh}, V.V.~Zhikov and S.E.~Pastukhova \cite{CPZh}, \cite{Pas}, \cite{PT}, \cite{Zh3}--
\cite{Zh4}, G. Griso \cite{Gr1}--\cite{Gr4}, and by C.E.~Kenig, F.~Lin, Z.~Shen \cite{KLS1}, \cite{KLS2}; see also other papers by these authors. Moreover, in the above cited works, the authors succeeded to establish sharp estimates for the rates of convergence in the sense of various operator norms.

In view of the above described results, a natural question appeared: whether a similar norm resolvent convergence is valid for other types of the perturbations in the homogenization theory? This issue was studied recently for certain perturbations in the boundary homogenization.

In \cite{PMA}, \cite{AHP}, \cite{CRAS}, \cite{ZAMP}, \cite{JPA09} problems with frequent alternation of boundary conditions were treated. The norm resolvent convergence was proven for all possible homogenized problems as well as for both periodic and non-periodic alternations. The estimates for the rate of   were obtained. In periodic cases certain asymptotic expansions for the spectra of perturbed operators were constructed.

In \cite{BCFP}, \cite{Na}, \cite[Ch. I\!I\!I, Sec. 4]{OIS} the norm resolvent convergence for problems with a fast periodically oscillating boundary was proven. The most general results was obtained in \cite{BCFP}. Namely, various geometries of oscillations as well as various boundary conditions on the oscillating boundary were considered. There were obtained estimates for the rate of norm resolvent convergence in the sense of various operator norms.

The norm resolvent convergence for  periodic perforations was studied in \cite{MCh2}, \cite{Pas}. In \cite{MCh2} the whole of a domain was perforated. The operator was described by the Helmholtz equation; on the boundaries of the holes the Dirichlet condition was imposed. The authors  treated the case when the holes disappeared under the homogenization and made no influence for the homogenized operators. The norm resolvent convergence was proven; no estimates for the rate of convergence were found. In   \cite{Pas}, an elliptic operator in
a perforated domain was studied. Here again whole of the domain was perforated. It was assumed that the sizes of the holes and the distances between them are of the same order of smallness. On the boundaries of the holes the Neumann condition was imposed. The norm resolvent convergence and the estimates for the rate of convergence were established.

One more interesting paper devoted to norm resolvent convergence is  \cite{ZhMIAN}. Here the perturbation was defined by rescaling an abstract periodic measure. The main result is the description of the homogenized operator, the proof of the norm resolvent convergence, and the estimates for the rate of convergence. A general model of \cite{ZhMIAN} covered various perturbations including periodic perforation of the whole of a domain
provided the sizes of the holes and the distances between them are of the same smallness order.

In the present paper we consider a general second order elliptic operator in an infinite planar strip perforated along a curve which is either infinite or finite and closed.  The sizes of the holes and the distance between them are described by means of two small parameters. The perforation is quite general and no periodicity is assumed. Namely, both the shapes and the distribution of the holes can be rather arbitrary. On the boundary of the holes we impose one of the classical boundary conditions, i.e., Dirichlet or Neumann or Robin condition. Boundaries of different holes can be subject to different types of boundary conditions.  To the best of authors' knowledge, such mixtures of boundary conditions were not considered before.

One of the possible physical interpretations of our operator comes from the waveguide theory. Namely, our operator describes a quantum particle in a waveguide modeled by an infinite strip, and since the coefficients of the operator are variable, the waveguide is not isotropic. The perforation can be interpreted as a series of small defects distributed along a given line, while the conditions on the  boundaries of the holes impose certain regime, for instance, the Dirichlet condition describes a wall and the particle can not pass through such boundary. Then  the homogenization  describes the effective behavior of our model once the perforation becomes finer, while the type of resolvent convergence characterizes in which sense the perturbed model is close to the effective one.

Our first main result describes the homogenized problems depending on the geometry, sizes, and distribution of the holes as well as of the  conditions on the boundary of the holes. The differential expression for the homogenized operator is the same as for the original operator, but with   Dirichlet  condition or delta-interaction or no condition on the reference curve along which the perforation is made. Our second main result is the the norm resolvent convergence of the perturbed operator to the homogenized one and the estimates for the rates of convergence. In all cases except one the operator norm is that of the operators from $L_2$ into $W_2^1$, while in the exceptional case it is from $L_2$ into $L_2$. Nevertheless, in the latter case we show that by employing a special boundary corrector one can replace the norm by that of the operators acting from $L_2$ into $W_2^1$. Such kind of results on norm resolvent convergence are completely new for the domains perforated periodically along curves or manifolds, especially in view of the fact that  we succeeded to study the general non-periodic perforation with arbitrary boundary conditions.

Our technique is based on the variational formulations of the equations for the perturbed and the homogenized operators. We   use no smoothing operator like, for instance, in the above cited papers on the operators with fast oscillating coefficients. Instead of this, we write the integral identity for the difference of the perturbed and homogenized resolvents and estimate then the terms coming from the boundary conditions. It requires certain accurate estimates for various boundary integrals over holes and over the reference curve. The main difference of our technique with that in the previous works is the assumptions for the perforation. In previous works \cite{art6}, \cite{art9}, \cite{art7}, 
\cite{art2}, \cite{art3}, \cite{art1}--
\cite{art5}, the main assumption was the existence of an operator of continuation the holes for the functions defined outside as well as uniform estimates for this operator. In our work, we assume the solvability of a certain fixed boundary value problems for the divergence operator in a neighborhood of the holes. We believe that our assumptions are not worse than the existence of the continuation operator since we require just a solvability of certain boundary value problem while the existence of the continuation operator means the possibility to extend \emph{each} function in a given Sobolev space.

Concluding the introduction, we describe briefly the structure of the paper. In the next section we give the precise description of the problem, formulate the main results, and discuss them. In the third section we collect auxiliary lemmata required for the proof of the main results. The forth, fifth, and sixth sections are devoted to the study of the norm resolvent convergence in various cases. In the seventh section we prove the convergence of the spectrum. In the last eighth section we discuss the sharpness of our estimates for the rate of convergence.

\section{Problem and main results}

Let $x=(x_1,x_2)$ be the Cartesian coordinates in $\mathds{R}^2$,
 $\Om:=\{x:
0<x_2<d\}$ be a horizontal strip of width $d>0$. By $\g$ we denote a curve in $\Om$ separated from $\p\Om$ by a fixed distance. Curve $\g$ is supposed to be  $C^3$-smooth and to have no self-intersections.
We consider two cases assuming that $\g$ is either an infinite curve or it is a finite closed curve. By $s$ we denote the arc length of $\g$, $s\in[-s_*,s_*]$, where $s_*$ is either finite or $s_*=+\infty$. If curve $\g$ is finite, we identify points $s=-s_*$ and $s=s_*$. By $\varrho=\varrho(s)$ we denote the vector function describing the curve $\g$. Since curve $\g$ is $C^3$-smooth, then $\varrho\in C^3[-s_*,s_*]$; for an infinite curve by $[-s_*,s_*]$ we mean $\mathds{R}$. The above assumptions for $\g$ yield that this curve partitions domain $\Om$ into two disjoint subdomains. The upper or exterior one is denoted by $\Om_+$ and the lower or interior subdomain is $\Om_-$. By $B_r(a)$ we denote the ball in $\mathds{R}^2$ of radius $r$ centered at $a$.

Let $\mathds{M}^\e\subseteq \mathds{Z}$ be some set, and  $s_k^\e\in [-s_*, s_*]$, $k\in\mathds{M}^\e$, be a set of points satisfying $s_k^\e<s_{k+1}^\e$. By $\om_k$,
$k\in\mathds{M}^\e$, we indicate a set of bounded domains in
$\mathds{R}^2$ having $C^2$-boundaries. We stress that these domains are not supposed to be simply connected. Denoting by $\e$ a small positive parameter, we define
\begin{align*}
&\tht^\e:=\tht^\e_0\cup\tht^\e_1,\quad
\tht_i^\e:=\bigcup\limits_{k\in\mathds{M}_i}\om_k^\e,\quad
i=0,1,
\quad \om_k^\e:=\{x: \e^{-1}\eta^{-1}(\e)(x-y_k^\e)\in\om_k\}, \quad y_k^\e:=\varrho(s_k^\e),
\end{align*}
where
$\mathds{M}_0^\e\cap\mathds{M}_1^\e=\emptyset$,
$\mathds{M}_0^\e\cup\mathds{M}_1^\e=\mathds{M}^\e$,
and $\eta=\eta(\e)$ is a some function satisfying the inequality $0<\eta(\e)\leqslant 1$.
We make the following assumptions.

\begin{enumerate}\def\theenumi{(A\arabic{enumi})} \setcounter{enumi}{0}

\item\label{B1} There exist fixed numbers $0<R_1<R_2$, $b>1$, $L>0$ and points $x^k\in\om_k$, $k\in\mathds{M}^\e$, such that
\begin{align*}
& B_{R_1}(x^k)\subset
\om_k\subset B_{R_2}(0),\quad |\p\om_k|\leqslant L \quad\text{for each}\quad k\in\mathds{M}^\e,
\\
&B_{b R_2\e}(y_k^\e)\cap B_{b R_2\e}(y_i^\e)=\emptyset\quad \text{for each} \quad i,k\in\mathds{M}^\e, \quad i\not=k,
\end{align*}
and for all sufficiently small $\e$.

\item \label{B2}
 For $b$ and $R_2$ in \ref{B1} and $k\in\mathds{M}^\e$ there exists a generalized  solution $X_k: B_{b_*R_2}(0)\setminus\om_k\mapsto \mathds{R}^2$, $b_*:=(b+1)/2$, to the boundary value problem
    \begin{equation}
    \begin{gathered}
    \mathrm{div}\, X_k=0\quad \text{in}\quad B_{b_*R_2}(0)\setminus\om_k,
    \\
    X_k\cdot\nu=-1\quad \text{on}\quad \p\om_k, \quad X_k\cdot\nu=\vp_k\quad \text{on}\quad \p B_{b_*R_2}(0),
    \end{gathered}\label{2.1a}
    \end{equation}
belonging to $L_\infty(B_{b_*R_2}(0)\setminus\om_k)$ and bounded in the sense of this space uniformly in  $k\in\mathds{M}^\e$. Here $\nu$ is the outward normal to $\p B_{b_*R_2}(0)$ and to $\p\om_k$, while $\vp_k$ is a some function in $L_\infty(\p B_{b_*R_2}(0))$ satisfying
\begin{equation}\label{2.26}
\int\limits_{\p B_{b_*R_2}(0)} \vp_k\di s=|\p\om_k|.
\end{equation}
\end{enumerate}

By $A_{ij}=A_{ij}(x)$, $A_i=A_i(x)$, $A_0=A_0(x)$ we denote   functions satisfying the conditions
\begin{equation}\label{2.3}
\begin{aligned}
&A_{ij}, A_i\in \Hinf^1(\Om), \quad i,j=1,2, \quad A_0\in L_\infty(\Om), \quad A_{ji}=A_{ji},
\\
&\sum\limits_{i,j=1}^{2} A_{ij}z_i z_j  \geqslant c_2|\xi|^2,\quad x\in\Om,\quad z=(z_1,z_2)\in \mathds{R}^2,
\end{aligned}
\end{equation}
where $c_2$ is a positive constant independent of $x$ and $\xi$, and $A_{ij}$, $A_0$ are real-valued.

In the vicinity of $\g$ we introduce local coordinates $(s,\tau)$, where $\tau$ is the distance to a point measured along the normal $\nu^0$ to $\g$ which is inward for $\Om_-$, and $s$, we remind, is the arc length of $\g$. Since the curvature of $\g$ is uniformly bounded, the coordinates $(s,\tau)$ are well-defined for $|\tau|\leqslant \tau_0$, $s\in \mathds{R}$, where $\tau_0$ is a sufficiently small fixed positive number.

We denote by $\Om^\e:=\Om\setminus\tht^\e$ our perforated domain, cf. Figure 1. In this
paper we study a singularly perturbed  operator depending on
$\e$ which we denote as $\ope$.  It is introduced by the differential expression
\begin{equation}\label{2.4}
-\sum\limits_{i,j=1}^{2} \frac{\p\hphantom{x}}{\p x_i} A_{ij} \frac{\p\hphantom{x}}{\p x_j} + \sum\limits_{j=1}^{2} A_j\frac{\p\hphantom{x}}{\p x_j}-\frac{\p\hphantom{x}}{\p x_j}\overline{A_j} + A_0
\end{equation}
in  $\Om^\e$ subject to the Dirichlet condition on $\p\Om\cup\p\tht^\e_0$ and to the Robin condition
\begin{equation*}
\left(\frac{\p\hphantom{N}}{\p N^\e}+a\right)u=0\quad\text{on}\quad \p\tht^\e_1,\qquad
\frac{\p\hphantom{N}}{\p N^\e}:=\sum\limits_{i,j=1}^{2} A_{ij}\nu_i^\e \frac{\p\hphantom{x}}{\p x_j}+\sum\limits_{j=1}^{2} \overline{A}_j\nu_j^\e,
\end{equation*}
where $\nu^\e=(\nu^\e_1,\nu^\e_2)$ is the inward normal to $\p\tht_1^\e$, $a=a(x)$ is a function defined for $|\tau|<\tau_0$ and
$a\in\Hinf^1(\{x: |\tau|<\tau_0\})$.

By $\fae$ we denote the sesquilinear form
\begin{equation}
\begin{aligned}
\fae(u,v):=&\sum\limits_{i,j=1}^{2} \left(A_{ij} \frac{\p u}{\p x_j}, \frac{\p v}{\p x_i}\right)_{L_2(\Om^\e)} + \sum\limits_{j=1}^{2} \left(A_j\frac{\p u}{\p x_j}, v\right)_{L_2(\Om^\e)}
\\
&+\sum\limits_{j=1}^{2}\left(u,A_j\frac{\p v}{\p x_j}\right)_{L_2(\Om^\e)}+(A_0 u,v)_{L_2(\Om^\e)}
\end{aligned}\label{2.4a}
\end{equation}
in $L_2(\Om^\e)$ on the domain $W_2^1(\Om^\e)$.
Rigorously we introduce operator $\ope$  as the lower-semibounded self-adjoint operator  in
$L_2(\Om^\e)$ associated with the closed lower-semibounded
symmetric sesquilinear form
$\foe(u,v):=\fae(u,v) +(au,v)_{L_2(\p\tht_1^\e)}$
in $L_2(\Om^\e)$ on $\Ho^1(\Om^\e,\p\Om\cup\p\tht^\e_0)$. Hereinafter for any domain $Q\subset \mathds{R}^2$ and any curve $S\subset Q$, by $\Ho^1(Q,S)$ we denote the subspace of $W_2^1(Q)$ consisting of the functions with zero trace on $S$, and we let $\Ho^1(Q):=\Ho^1(Q,\p Q)$. If else is not said, in what follows all the differential operators are introduced in this way, i.e., they will be self-adjoint lower semibounded operators in $L_2(\Om)$ or $L_2(\Om^\e)$ associated with closed lower-semibounded symmetric sesquilinear form. For the sake of brevity, for each operator we shall just write the differential expression with the boundary condition as well as the associated form.
\begin{figure}[t]
\centering
\begin{tabular}{cc}
\includegraphics[scale=0.3]{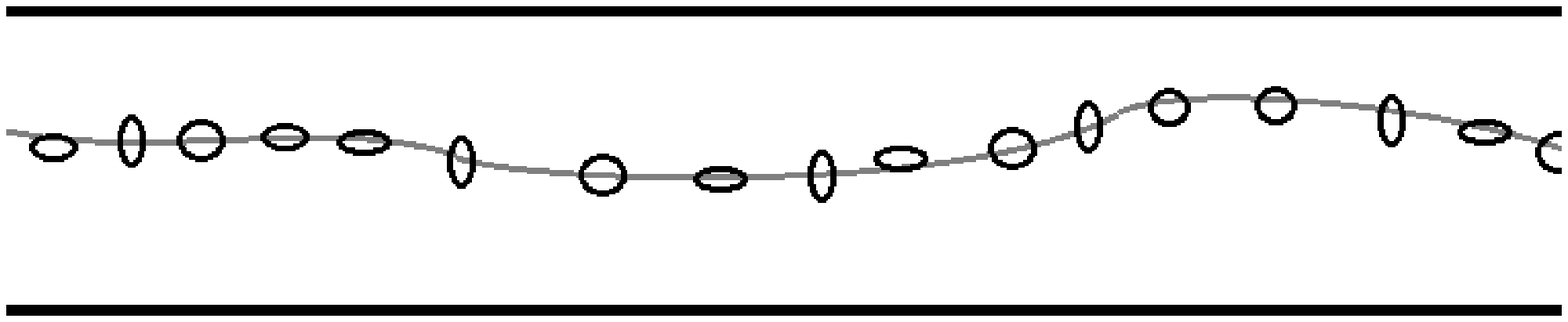} & \includegraphics[scale=0.3]{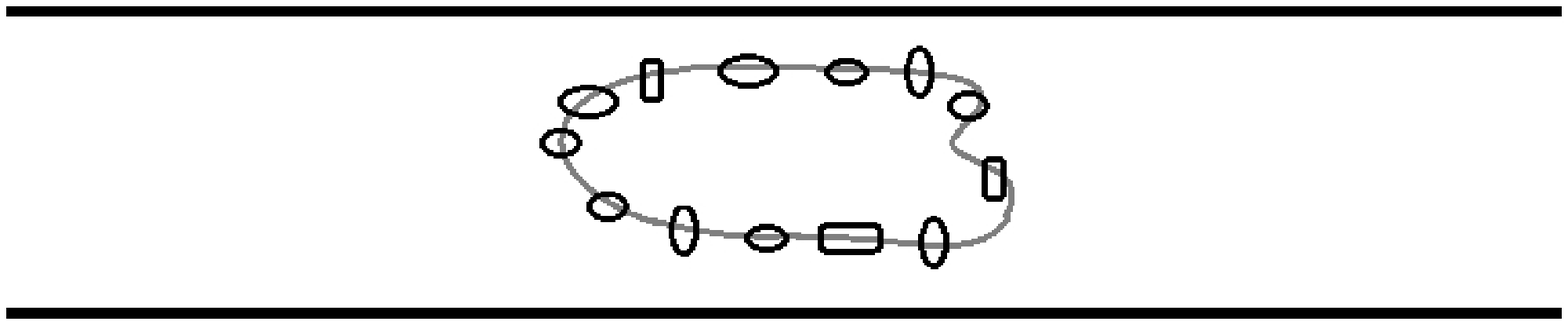}
\\
\
\\
{\small (a) Perforation along an infinite curve} & {\small (b) Perforation along a closed curve}
\end{tabular}

\caption{Perforated domain}
\end{figure}

Our main aim is to study the resolvent convergence and the spectrum's behavior of the operator $\ope$. To formulate our main results, we need additional notations.

By $\opo{D}$ we denote the operator in $L_2(\Om)$ with the differential expression (\ref{2.4}) subject to the Dirichlet condition on $\g$ and $\p\Om$. The associated form is
$\foo{D}(u,v):=\fao(u,v)$
in $L_2(\Om)$ on $\Ho^1(\Om,\p\Om\cup\g)$, where form $\fao$
is introduced by expression (\ref{2.4a}), where $\Om^\e$ is replaced by $\Om$.
By analogy with \cite[Lem. 2.2]{BorAA}, \cite[Ch. I\!V, Sec. 2.2, 2.3]{Mi}, \cite[Lem. 3.2]{BorIEOP} one can check that the domains of
operator $\opo{D}$
is given by the identity
$\Dom(\opo{D})=\Ho^1(\Om,\p\Om\cup\g)\cap W_2^2(\Om\setminus\g)$.
By $\iu$ we denote the imaginary unit, and by $\|\cdot\|_{X\to Y}$ we denote the norm of an operator acting from a Banach space $X$ to a Banach space $Y$.

Now we are ready to formulate our first main result.

\begin{theorem}\label{th2.1} Let
\begin{equation}\label{2.1ab}
\e\ln\eta(\e)\to0,\quad \e\to+0,
\end{equation}
and suppose \ref{B1}, \ref{B2}, and
\begin{enumerate}
\def\theenumi{(A\arabic{enumi})} \setcounter{enumi}{2}

\item \label{A5}
There exists a constant $R_3>b R_2$ such that
\begin{equation*}
\{x: |\tau|<\e b R_2\} \subset\bigcup\limits_{k\in\mathds{M}_0^\e} B_{R_3\e}(y_k^\e),\quad\om_k^\e\subset B_{R_3\e}(y_k^\e)  \quad \text{for each} \quad k\in\mathds{M}_0^\e.
\end{equation*}
\end{enumerate}

\noindent Then  the estimate
\begin{equation}\label{2.5a}
\|(\ope-\iu)^{-1}-(\opo{D}-\iu)^{-1}\|_{L_2(\Om)\to W_2^1(\Om^\e)}\leqslant C \e^{\frac{1}{2}}\big(|\ln\eta(\e)|^{\frac{1}{2}}+1\big)
\end{equation}
holds true, where $C$ is a positive constant independent of $\e$.
\end{theorem}

Let $\nu^0=(\nu^0_1,\nu^0_2)$   and
\begin{equation*}
\hphantom{\Bigg(}\frac{\p\hphantom{N}}{\p N^0}:=\sum\limits_{i,j=1}^{2} A_{ij}\nu_i^0 \frac{\p\hphantom{x}}{\p x_j}.
\end{equation*}
By $[\cdot]_\g$ we indicate the jump of a function on $\g$,
$[u]_\g=u\big|_{\tau=+0}-u\big|_{\tau=-0}$.
Given a function $\beta=\beta(s)$ in $\Hinf^1(\g)$, we introduce   operator $\opo{\beta}$ with  differential expression (\ref{2.4}) subject to the boundary conditions
\begin{equation}\label{2.13}
[u]_\g=0,\quad \bigg[\frac{\p u}{\p N^0}\bigg]_\g+\beta u\big|_\g=0.
\end{equation}
The associated   form is
$\foo{\beta}(u,v):=\fao(u,v)
+(\beta u,v)_{L_2(\g)}$
in $L_2(\Om)$ on $\Ho^1(\Om)$. Again by analogy with \cite[Lem. 2.2]{BorAA}, \cite[Ch. I\!V, Sec. 2.2, 2.3]{Mi}, \cite[Lem. 3.2]{BorIEOP} one can show that
\begin{equation*}
\Dom(\opo{\beta})=\{u\in\Ho^1(\Om):  u\in W_2^2(\Om_\pm)\text{ and (\ref{2.13}) is satisfied}\}.
\end{equation*}
If $\beta=0$, instead of $\opo{0}$ we shall simply write $\opo{}$. As one can see, in this case there is no boundary condition on $\g$ and  the domain of $\opo{}$ is  $\Dom(\opo{})=\Ho^1(\Om)\cap W_2^2(\Om)$.

In the next theorem we deal with the case when the perturbed operator involves the Dirichlet condition at least on a part of $\p\tht^\e$ but in contrast to (\ref{2.1ab}), the function  $\e\ln\eta(\e)$ converges either to a non-zero constant or to infinity.

\begin{theorem}\label{th2.4}
Suppose \ref{B1}, \ref{B2}, let
\begin{equation}\label{2.12}
\frac{1}{\e\ln\eta(\e)}\to-\rho,\quad \e\to+0,
\end{equation}
and set $\mathds{M}_0^\e$ be non-empty. For
$b$ and $R_2$ in \ref{B1} and $s\in \mathds{R}$ we denote
\begin{equation*}
\alpha^\e(s):=\left\{
\begin{aligned}
&  \frac{\pi}{b R_2}, \quad  |s-s_k^\e|<b R_2\e,\quad k\in\mathds{M}^\e_0,
\\
&\hphantom{0} 0, \hphantom{000} \quad \text{otherwise}.
\end{aligned}\right.
\end{equation*}
Suppose also that

\begin{enumerate}\def\theenumi{(A\arabic{enumi})} \setcounter{enumi}{3}

\item\label{A14}
There exists a function $\alpha=\alpha(s)$ in $\Hinf^1(\g)$ and a function $\varkappa=\varkappa(\e)$, $\varkappa(\e)\to+0$, $\e\to+0$, such that for  all sufficiently small $\e$ the  estimate
\begin{equation}\label{2.17}
\sum\limits_{q\in\mathds{Z}} \frac{1}{|q|+1} \left|\int\limits_{n}^{n+\ell}
  \big(\a^\e(s)-\a(s)\big)\E^{-\frac{\iu q}{2\pi\ell}(s-n)}\di s\right|^2
\leqslant \varkappa^2(\e)
\end{equation}
is valid, where $n=-s_*$, $\ell=2 s_*$, if $\g$ is a finite curve, and $n\in \mathds{Z}$, $\ell=1$, if $\g$ is an infinite curve. In the latter case estimate (\ref{2.17}) is supposed to hold uniformly in $n$.
\end{enumerate}

\noindent Denote
\begin{equation*}
\beta:=\alpha\frac{(\rho+\mu)}{A_{11}A_{22}-A_{12}^2},\quad \beta_0:=\alpha\frac{\rho}{A_{11}A_{22}-A_{12}^2},\quad \mu(\e):=-\frac{1}{\e\ln\eta(\e)}-\rho.
\end{equation*}
Then  the estimates
\begin{align}
&\|(\ope-\iu)^{-1}-(\opo{\beta}-\iu)^{-1}\|_{L_2(\Om)\to L_2(\Om^\e)}\leqslant C\big(\e^{\frac{1}{2}}+\varkappa (\e)\big) \label{2.18a}
\\
&\|(\ope-\iu)^{-1} -(\opo{\beta_0}-\iu)^{-1} \|_{L_2(\Om)\to L_2(\Om^\e)}\leqslant C\big(\e^{\frac{1}{2}}+\varkappa (\e)+\mu(\e)\big) \label{2.19}
\end{align}
hold true, where $C$ is a positive constant independent of $\e$. There exists an explicit function $W^\e$ defined in (\ref{7.31}) such that the estimate
\begin{equation}\label{2.21}
\|(\ope-\iu)^{-1} -(1-W^\e)(\opo{\beta}-\iu)^{-1} \|_{L_2(\Om)\to W_2^1(\Om^\e)}\leqslant C\big(\e^{\frac{1}{2}}+\varkappa(\e)(\rho+\mu(\e))\big)
\end{equation}
is valid, where $C$ is a positive constant independent of $\e$. If $\rho=0$, the estimate
\begin{equation}\label{2.20}
\|(\ope-\iu)^{-1} -(\opo{}-\iu)^{-1} \|_{L_2(\Om)\to W_2^1(\Om^\e)}\leqslant C\big(\e^{\frac{1}{2}}+\mu^{\frac{1}{2}}(\e)\big)
\end{equation}
holds true, where $C$ is a positive constant independent of $\e$.
\end{theorem}

The next two theorems concern the case when $\mathds{M}_0^\e$ is empty, i.e., the perturbed operator involves just the Robin condition on $\p\tht^\e$.

\begin{theorem}\label{th2.2}
Suppose \ref{B1}, \ref{B2}, let set $\mathds{M}_0^\e$ be empty and
either $a\equiv 0$ or $\eta(\e)\to0$, $\e\to+0$.
Then the estimate
\begin{equation}
\|(\ope-\iu)^{-1} -(\opo{}-\iu)^{-1} \|_{L_2(\Om)\to W_2^1(\Om^\e)}\leqslant C \eta(\e)|\ln\eta(\e)|^{\frac{1}{2}}, \label{2.13a}
\end{equation}
holds true, if $a\not\equiv0$, $\eta\to+0$, and
\begin{equation}
\|(\ope-\iu)^{-1}f-(\opo{}-\iu)^{-1}f\|_{L_2(\Om)\to W_2^1(\Om^\e)}\leqslant C \e^{\frac{1}{2}}\eta(\e)(|\ln\eta(\e)|^{\frac{1}{2}}+1), \label{2.13b}
\end{equation}
if $a\equiv 0$. Here $C$ is a positive constant independent of $\e$.
\end{theorem}

\begin{theorem}\label{th2.3}
Suppose \ref{B1}, \ref{B2},  let $\eta=\mathrm{const}$, set $\mathds{M}_0^\e$ be empty. For $b$ and $R_2$ in \ref{B1} we denote
\begin{equation*}
\alpha^\e(s):=\left\{
\begin{aligned}
& \frac{|\p\om_k|\eta}{2b R_2}, \quad  |s-s_k^\e|<b R_2\e,\quad k\in\mathds{M}^\e,
\\
&\hphantom{000} 0, \hphantom{00} \quad \text{otherwise}.
\end{aligned}\right.
\end{equation*}
Suppose also that
\begin{enumerate}\def\theenumi{(A\arabic{enumi})} \setcounter{enumi}{4}

\item\label{A10}
There exists a function $\alpha=\alpha(s)$ in $\Hinf^1(\g)$ and a function $\varkappa=\varkappa(\e)$, $\varkappa(\e)\to+0$, $\e\to+0$, such that for all sufficiently small $\e$ the  estimates (\ref{2.17})
are valid.
\end{enumerate}

\noindent
Then the estimate
\begin{equation}\label{2.13c}
\|(\ope-\iu)^{-1} -(\opo{\alpha a}-\iu)^{-1} \|_{L_2(\Om)\to W_2^1(\Om^\e)}\leqslant C\big(\e^{\frac{1}{2}}+\varkappa(\e)\big)
\end{equation}
holds true, where $C$ is a positive constant independent of $\e$.
\end{theorem}

Let us discuss the main results. Assumption \ref{B1} says that the sizes of holes are of the same order and there is a minimal distance between them. This is a very natural assumption for the perforation. At the same time, no periodicity for the perforation is assumed. Moreover, since set $\mathds{M}^\e$ is arbitrary, we do not need to assume that it is infinite, and for instance, the number of holes can be finite. In the latter case, by choosing an appropriate set $\mathds{M}^\e$, we can even get the situation when the distances between the holes are not small, but finite. In this situation one can still apply Theorems~\ref{th2.4}--\ref{th2.3}. Theorem~\ref{th2.1} is valid only in the case when the holes with Dirichlet condition are distributed quite densely in order to satisfy Assumption \ref{A5}.

Assumption \ref{B2} is a restriction for the geometry of boundaries $\p\om_k$. We first stress that problem (\ref{2.1a}) can be rewritten to the Neumann problem for the Laplace equation by letting $X_k=\nabla V_k$. Then identity (\ref{2.26}) is the solvability condition and this is the only restriction for $\vp_k$ we suppose. Problem (\ref{2.1a}) is solvable for each fixed $k$ and it is solution belongs to $L_\infty(B_{b_* R_2}(0)\setminus\om_k)$. And we assume that the norm $\|X_k\|_{L_\infty(B_{b_* R_2}(0)\setminus\om_k)}$ is bounded uniformly in $k$.

According to Theorem~\ref{th2.1}, if the sizes of the holes are not too small (cf. (\ref{2.1ab})) and the holes with the Dirichlet condition are, roughly speaking, distributed ``uniformly'' (Assumption \ref{A5}), the homogenized operator is subject to the Dirichlet condition on $\g$ and we have the norm resolvent condition in the sense of the operator norm $\|~\cdot~\|_{L_2(\Om)\to L_2(\Om^\e)}$. As one can see, relation (\ref{2.1ab}) admits the situation when the sizes of the holes are much smaller than the distances between them (for instance, $\eta(\e)=\e^\alpha$, $\alpha=const>0$), but nevertheless the homogenized operator is still subject to the Dirichlet condition on $\g$. This phenomenon is close to a similar one for the operators with frequent alternation of boundary conditions, cf.
\cite{ZAMP}, \cite{JPA09}, \cite{Ch1}.

If the function $\e\ln\eta(\e)$ goes to a constant or to infinity as $\e\to+0$ and there are holes with the Dirichlet condition, the homogenized operator has boundary condition (\ref{2.13}) on $\g$, see Theorem~\ref{th2.4}. This boundary condition describes a delta-interaction on $\g$, see, for instance, \cite[App. K, Sec. K.4.1]{Al},  and  the similar situation holds for the problems with frequent alternation of boundary conditions with the Dirichlet conditions on exponentially small parts of the boundary, cf. \cite{Ch1}, \cite{AHP}, \cite{CRAS}, \cite{ZAMP}. The norm resolvent convergence holds in the sense of the operator norm $\|\cdot\|_{L_2(\Om)\to L_2(\Om^\e)}$ only. To improve the norm, one  has either to employ the boundary corrector, see (\ref{2.21}), or to assume additionally $\rho=0$, see (\ref{2.20}). We observe that according to Assumption \ref{A14}, coefficient $\beta$ in boundary condition (\ref{2.13}) for the homogenized operator depends only on the distribution of the points $s_k^\e$ and there is no dependence on the geometries of the holes. There are also no special restrictions for part $\p\tht^\e_0$ with the Dirichlet condition. For instance, the number of holes in $\p\tht^\e_0$ can be finite or infinite and the distribution of this set can be very arbitrary.

If the perturbed operator has no Dirichlet condition on $\p\tht^\e$, the homogenized operator has either condition (\ref{2.13}) on $\g$ (Theorem~\ref{th2.3}) or even no condition (Theorem~\ref{th2.2}). In both cases we again have the norm resolvent convergence in the operator norm $\|\cdot\|_{L_2(\Om)\to W_2^1(\Om^\e)}$.
In Theorem~\ref{th2.2} we need no additional restrictions thanks to the assumption $\eta(\e)\to+0$ or $a\equiv 0$. In Theorem~\ref{th2.3} $\eta$ is constant and because of this we introduce Assumption \ref{A10}. Its means that the lengths of $\p\om_k$ should be distributed rather smoothly to satisfy (\ref{2.17}). We stress that the coefficient $\beta$ in (\ref{2.13}) for the homogenized operator depends both on the distribution of the holes and the sizes of their boundaries.

Let us also discuss assumptions \ref{A14} and \ref{A10}. This is in fact the same assumption but adapted for two different cases. The sum in the left hand side of (\ref{2.17}) is nothing but the norm in $W_2^{-\frac{1}{2}}(0,\ell)$.  This estimate obviously holds true for a periodic perforation. As an example of a non-periodic perforation, we can mention the situation when we start with a strictly periodic perforation along an infinite curve but then we change the geometry and locations of a part of holes so that the total number of  deformed holes associated with each segment $s\in(q,q+1)$, $q\in\mathds{Z}$, is relatively small in comparison with unchanged holes. Then inequality (\ref{2.17}) is still true. Moreover, our conjecture is that Assumptions \ref{A14} and \ref{A10} can not be improved or omitted once we want to have a norm resolvent convergence. Namely, in the proofs of Theorems~\ref{th2.2},~\ref{th2.3} these assumptions are employed only in Lemma~\ref{lm6.3}  and all the inequalities in the proof of this lemma are sharp.
There is also another way  of simplifying (\ref{2.17}) which is estimating $W_2^{-\frac{1}{2}}(0,\ell)$-norm by  $L_2(0,\ell)$-norm. Then (\ref{2.17}) can be replaced by
\begin{equation*}
\|\a^\e-\a\|_{L_2(n,n+\ell)}^2
\leqslant C \|\a^\e-\a\|_{L_1(n,n+\ell)}\leqslant \varkappa^2(\e),
\end{equation*}
where we have employed the boundedness of $\a^\e$, see Lemma~\ref{lm6.3}. However, this condition happens to be too restrictive and is satisfied just by few examples.

Last but not the least issue related to the above theorems is  the sharpness of the estimates for the rate of convergences. Many of these estimates are order sharp, i.e., the smallness order can not be improved. At the same time, the study of the sharpness is an independent problem that requires a completely different approach in comparison with the technique we employ in the proofs of Theorems~\ref{th2.1}--\ref{th2.3}. This is why we formulate no statement on the sharpness in the theorems and we discuss this issue independently in Section~8.

Our final main result describes the convergence of the spectrum of $\ope$.

\begin{theorem}\label{th2.5}
Under the hypotheses of Theorems~\ref{th2.1}--\ref{th2.3} the spectrum of perturbed operator $\ope$ converges to that of the corresponding homogenized operator. Namely, if $\lambda$ is not in the spectrum of the homogenized operator, for sufficiently small $\e$ the same is true for the perturbed operator. And if $\lambda$ is in the spectrum of the homogenized operator, for each $\e$ there exists $\lambda_\e$ in the spectrum of the perturbed operator such that $\lambda_\e\to\lambda$ as $\e\to+0$.
\end{theorem}

We note that this theorem is not implied immediately by Theorems~\ref{th2.1}--\ref{th2.3}. Despite these theorems state convergence of the perturbed resolvent  to a homogenized one in the norm sense, the norm is $\e$-dependent. Nevertheless, this makes no serious troubles and in the proof of Theorem~\ref{th2.5} we demonstrate a simple trick to overcome this difficulty.

Throughout the rest of the paper we shall indicate by $C$, $C_1$, $C_2$, $C_3$, \ldots various inessential positive constants independent of $\e$, $\eta(\e)$, $s$, $\tau$, $x$, and various functions $f$, $u$, $v$, \ldots\  from Sobolev spaces we shall deal with. In all the estimates such constants are independent on the functions written explicitly. In the case of local estimates in a vicinity of each $\om_k^\e$ such constants are also supposed to be independent of $k$. If these constants depend on some auxiliary parameters, it will be indicated explicitly. We shall also make use of  auxiliary notations: $B^k_r:=B_{r R_2\e\eta}(y_k^\e)$, $\widehat{B}^k_r:=B_{r R_2\e\eta}(y_k^\e)\setminus\om_k^\e$.

\section{Preliminaries}

In this section we collect several auxiliary lemmata which will be employed in the proof of our main results in the subsequent sections. These lemmata provide some estimates for various norms in Sobolev spaces as well as some local estimates in the vicinity of holes $\om_\e^k$.

In all the lemmata we assume \ref{B1}, \ref{B2}.

\begin{lemma}\label{lm1.6}
For any $\d>0$ there exists a constant $C(\d)>0$ such that
the estimates
\begin{align}
&\|u\|_{L_2(\g)}^2\leqslant \d\|\nabla u\|_{L_2(\Om)}^2+C(\d)\|u\|_{L_2(\Om)}^2,\nonumber
\\
&\|v\|_{L_2(\{x: \tau=-(b+1)R_2\e\})}^2\leqslant \d\|\nabla v\|_{L_2(\Om^\e)}^2+C(\d)\|v\|_{L_2(\Om^\e)}^2,\label{3.15}
\end{align}
are valid for each $u\in W_2^1(\Om)$, $v\in W_2^1(\Om^\e)$.
\end{lemma}

The statement of this lemma follows from \cite[Ch. I\!I, Sec. 2,  Ineq. (2.38)]{Ld}.

For $c\leqslant \tau_0$ we denote $\Pi^c:=\{x: |\tau|<c\}$.

Next four lemmata provide local estimates in the vicinity of the holes.

\begin{lemma}\label{lm3.3b}
For any $\d>0$ there exists a constant $C(\d)>0$ such that for each $v\in W_2^1(\Om^\e)$ the uniform estimates
\begin{align}
\sum\limits_{k\in\mathds{M}^\e} \|v\|_{L_2(\widehat{B}^k_{b_*})}^2 &\leqslant C \bigg(\e\eta \sum\limits_{k\in\mathds{M}^\e} \|v\|_{L_2(\p B^k_{b_*})}^2
+ \e^2\eta^2 \sum\limits_{k\in\mathds{M}^\e} \|\nabla v\|_{L_2(\widehat{B}^k_{b_*})}^2
\bigg),
\label{3.7b}
\\
\sum\limits_{k\in\mathds{M}^\e} \|v\|_{L_2(\widehat{B}^k_{b_*})}^2 &\leqslant \e\eta^2(|\ln\eta|+1) \Big(\d\|\nabla v\|_{L_2(\Om^\e)}^2+C(\d)\|v\|_{L_2(\Om^\e)}^2\Big),\label{3.11}
\\
\|v\|_{L_2(\p\tht^\e)}^2&\leqslant \eta(|\ln\eta|+1)  \Big(\d\|\nabla v\|_{L_2(\Om^\e)}^2+C(\d)\|v\|_{L_2(\Om^\e)}^2\Big)\label{3.12}
\end{align}
hold true for sufficiently small $\e$.
\end{lemma}

\begin{proof}
Assume first that $v\in C^1(\overline{\Om^\e})$. Thanks to Assumption \ref{B2}, in the integral identity for $X_k^\e(x):=X_k(z^k\e^{-1}\eta^{-1})$, $\z^k=(\z^k_1, \z^k_2):=x-y_k^\e$, we can take $|v|^2$ as the test-function,
\begin{equation}\label{3.10b}
\begin{aligned}
\|v\|_{L_2(\p\om_k^\e)}^2-
(\vp_k^\e v,v)_{L_2(\p B^k_{b_*})}
= -\int\limits_{\widehat{B}^k_{b_*}} X_k^\e\cdot\nabla |v|^2\di x,\quad \vp_k^\e:=\vp_k(\z^k\e^{-1}\eta^{-1}).
\end{aligned}
\end{equation}
By Assumptions \ref{B1}, \ref{B2}  we thus obtain
\begin{equation}\label{3.5b}
\begin{aligned}
  \|v\|_{L_2(\p\om_k^\e)}^2\leqslant C \bigg(&  \|v\|_{L_2(\p B^k_{b_*})}^2
  +   \|v\|_{L_2( \widehat{B}^k_{b_*})}
 \|\nabla v\|_{L_2( \widehat{B}^k_{b_*})}\bigg).
\end{aligned}
\end{equation}
Since $C^1(\overline{\Om^\e})$ is dense in $W_2^1(\Om^\e)$, the latter estimate is valid for each $v\in W_2^1(\Om^\e)$.

Let $\nu_1$ be the first coordinate of normal vector $\nu$ to $\p\om^\e_k$. We integrate by parts:
\begin{equation*}
\int\limits_{\widehat{B}^k_{b_*}} |v|^2\di x= 
 \int\limits_{\widehat{B}^k_{b_*}} |v|^2\frac{\p \z^k_1}{\p x_1}\di x
=\int\limits_{\p B^k_{b_*}} |v|^2 \frac{(\z^k_1)^2}{|\z^k|}\di s - \int\limits_{\p\om^\e_k} |v|^2\z^k_1\nu_1\di s
 - \int\limits_{\widehat{B}^k_{b_*}} \z^k_1\frac{\p|v|^2}{\p x_1}\di x.
\end{equation*}
As in (\ref{3.10b}), (\ref{3.5b}), we first make the integration for $v\in C^1(\overline{\Om^\e})$ and then we extend the resulting idenity for each $v\in W_2^1(\Om^\e)$. It implies
\begin{equation*}
\|v\|_{L_2(\widehat{B}^k_{b_*})}^2\leqslant C
 \e\eta \bigg( \|v\|_{L_2(\p B^k_{b_*})}^2 + \|v\|_{L_2(\p\om^\e_k)}^2
  + \|v\|_{L_2(\widehat{B}^k_{b_*})}\|\nabla v\|_{L_2(\widehat{B}^k_{b_*})}
\bigg).
\end{equation*}
This estimate, (\ref{3.5b}), and Cauchy-Schwarz inequality yield
\begin{align}
&\|v\|_{L_2(\widehat{B}^k_{b_*})}^2 \leqslant C \e\eta\Big(\|v\|_{L_2(\p B^k_{b_*})}^2+\e \eta
\|\nabla v\|_{L_2(\widehat{B}^k_{b_*})}^2\Big),\label{3.18}
\\
&\|v\|_{L_2(\p\om_k^\e)}^2\leqslant C\Big( \|v\|_{L_2(\p B^k_{b_*})}^2 + \e\eta \|\nabla v\|_{L_2(\widehat{B}^k_{b_*})}^2
\Big).\label{3.14}
\end{align}
Summing up  inequality (\ref{3.18}) w.r.t. $k$, we arrive at (\ref{3.7b}).

Let us  estimate $\|v\|_{L_2(\p B^k_*)}$. In  order to do it, we follow the ideas employed in the proof of Lemma~3.2 in \cite{OSHY}. We introduce an infinitely differentiable cut-off function $\chi_1=\chi_1(t)$ being one as $t<1$ and vanishing as $t>2$. We have
\begin{equation*}
v(x)=v(x)\chi_1\left(\frac{|\z^k|R_2^{-1}\e^{-1}-1}{b-1}\right)\quad\text{as}\quad x\in B^k_{b}\setminus B^k_{1}.
\end{equation*}
Let $(r,\vp)$ be polar coordinates centered at $y_k^\e$. By  Assumption \ref{B1}, the ball $B_{(2b-1)R_2\e}(y_k^\e)$ does not intersect with $\om_i^\e$, $i\not=k$. Hence, for $R_2\e\eta\leqslant r\leqslant b R_2\e\eta$ Cauchy-Schwarz inequality implies
\begin{align*}
|v(r,\vp)|^2&= \left|\int\limits^{(2b-1)R_2\e}_r \frac{\p}{\p t}\left(v(t,\vp) \chi_1\left(\frac{t R_2^{-1}\e^{-1}-1}{b-1}\right)\right)\di t\right|^2
\\
&\leqslant \int\limits_{r}^{(2b-1)R_2\e} \frac{\di t}{t} \int\limits_{r}^{(2b-1)R_2\e} \left|\frac{\p}{\p t}\left(v(t,\vp)\chi_1\left(\frac{t R_2^{-1}\e^{-1}-1}{b-1}\right)
\right) \right|^2t\di t
\\
&\leqslant C(|\ln\eta|+1) \int\limits_{R_2\e\eta}^{(2b-1)R_2\e} \left( \left|\frac{\p v}{\p t}(t,\vp)\right|^2+\e^{-2}
|v(t,\vp)|^2
\right)t\di t.
\end{align*}
We integrate this estimate over $\p B^k_{b_*}$ and get
\begin{equation}\label{3.20}
\|v\|_{L_2(\p B^k_{b_*})}^2 \leqslant C\e\eta(|\ln\eta|+1) \Big(
\|\nabla v\|_{L_2(B_{(2b-1)R_2\e}(y_k^\e)\setminus B^k_{1})}^2
 + \e^{-2} \|v\|_{L_2(B_{(2b-1)R_2\e}(y_k^\e)\setminus B^k_{1})}^2\Big).
\end{equation}
To estimate the last term in the right hand side of the obtained inequality, we observe that
for sufficiently small $\d$ and $|\tau|<\tau_0/4$,
\begin{equation*}
|v(\tau,s)|^2=\int\limits_{\pm\frac{\d\tau_0}{2}}^{\tau} \frac{\p}{\p t} \left(|v(\tau,s)|^2\chi_1\left(\frac{4|t|}{\d\tau_0}\right)\right)\di t,\quad \pm\tau>0,
\end{equation*}
and by Cauchy-Schwarz inequality
\begin{equation}\label{3.4a}
|v(\tau,s)|^2\leqslant \frac{\d^2\tau_0^2}{2} \int\limits_{\pm\frac{\d\tau_0}{2}}^{\tau} \left|\frac{\p v}{\p\tau}(t,s)\right|^2\di t+ C(\d) \int\limits_{\pm\frac{\d\tau_0}{2}}^{\tau} |v(t,s)|^2\di t,\quad \pm\tau>0.
\end{equation}
Integrating this estimate over $\Pi^{2b R_2\e}\setminus\bigcup\limits_{k\in\mathds{M}^\e} B^k_1$, we have
3.21\begin{equation}\label{3.21}
\|v\|_{L_2(\Pi^{2b R_2\e}\setminus\bigcup\limits_{k\in\mathds{M}^\e} B^k_1)}^2 \leqslant \e \Big(\d\|\nabla v\|_{L_2(\Om^\e)}^2 + C(\d)\|v\|_{L_2(\Om^\e)}^2\Big).
\end{equation}
Now we substitute (\ref{3.20}) into (\ref{3.18}), sum up the result over $k\in\mathds{M}^\e$, and apply then (\ref{3.21}). It leads us to (\ref{3.11}). Inequality (\ref{3.12}) follows from (\ref{3.14}), (\ref{3.20}), (\ref{3.21}).
\end{proof}

\begin{lemma}\label{lm3.7}
For each $u,v\in W_2^1(\Om^\e)$ the uniform estimate
\begin{align*}
\sum\limits_{k\in\mathds{M}^\e} \bigg| \frac{|\p\om_k|}{\pi(b+1)R_2} (au,v)_{L_2(\p B^k_{b_*})}  &- (au,v)_{L_2(\p\om_k^\e)}
\bigg|
\leqslant C\e^{\frac{1}{2}}\eta \big(|\ln\eta|^{\frac{1}{2}}+1\big) \|u\|_{W_2^1(\Om^\e)} \|v\|_{W_2^1(\Om^\e)}
\end{align*}
holds true for sufficiently small $\e$.
\end{lemma}
\begin{proof}
In the same way how (\ref{3.10b}), (\ref{3.5b}) were proven, one can check easily one more estimate
\begin{equation}\label{3.23}
\sum\limits_{k\in\mathds{M}^\e} \bigg|
(a u,v)_{L_2(\p\om_k^\e)}-
(a\vp_k^\e u,v)_{L_2(\p B^k_{b_*})}
\bigg|
\leqslant  C\e^{\frac{1}{2}}\eta \big(|\ln\eta|^{\frac{1}{2}}+1\big) \|u\|_{W_2^1(\Om^\e)} \|v\|_{W_2^1(\Om^\e)}.
\end{equation}

Denote
\begin{equation*}
\la v \ra_k:=\frac{4}{\pi(3b^2-2b-1)\e^2\eta^2} \int\limits_{B^k_{b}\setminus B^k_{b_*}} v\di x, \quad v^\bot:=v-\la v \ra_k.
\end{equation*}
It is clear that $\int\limits_{B^k_{b}\setminus B^k_{b_*}} v^\bot \di x=0$. Then we rescale the variables $x\mapsto \z^k\e^{-1}\eta^{-1}$ and employ the Poincar\'e inequality
\begin{equation*}
\|\psi\|_{L_2(\p B_{b_*R_2}(0))}\leqslant C\|\nabla \psi\|_{L_2(B_{b R_2}(0)\setminus B_{b_*R_2}(0))},
\end{equation*}
which is valid for each $\psi\in W_2^1(B_{b R_2}(0)\setminus B_{b_*R_2}(0))$ satisfying $\int\limits_{B_{b R_2}(0)\setminus B_{b_*R_2}(0))}\psi\di x=0$. It leads us to the estimate
\begin{equation*}
\|v^\bot\|_{L_2(\p B_{b_*}^k)} \leqslant C\e^{\frac{1}{2}}\eta^{\frac{1}{2}} \|\nabla v\|_{L_2(B_{b}^k\setminus B_{b_*}^k)},
\end{equation*}
and the same is valid for $(au)^\bot:=au-\la au\ra_k$. Hence, we have
\begin{equation*}
(a\vp_k^\e u,v)_{L_2(\p B_{b_*}^k)} =
 \la au \ra_k \la v \ra_k \int\limits_{\p B_{b_*}^k} \vp_k^\e \di s
+\big(\vp_k^\e (au)^\bot, v\big)_{L_2(\p B_{b_*}^k)} + (\vp_k^\e, v^\bot)_{L_2(\p B_{b_*}^k)} \la au \ra_k.
\end{equation*}
Since by \ref{B2}
\begin{equation*}
 \int\limits_{\p B_{b_*}^k} \vp_k^\e \di s=\e\eta \int\limits_{\p B_{b_*R_2}(0)} \vp_k\di s=\e\eta|\p\om_k|,
\end{equation*}
it follows from (\ref{3.11}), (\ref{3.12}) with $\om_k=B_{b_*R_2}(0)$, (\ref{3.21}) that
\begin{align*}
\sum\limits_{k\in\mathds{M}^\e} \left|(a\vp_k^\e u,v)_{\p B_{b_*}^k} - \e\eta |\p\om_k| \la au \ra_k \la v \ra_k \right|\leqslant  C\e^{\frac{1}{2}}\eta \big(|\ln\eta|^{\frac{1}{2}}+1\big) \|u\|_{W_2^1(\Om^\e)} \|v\|_{W_2^1(\Om^\e)}.
\end{align*}
Completely in the same way one can show  that
\begin{align*}
\sum\limits_{k\in\mathds{M}^\e} \bigg|\frac{|\p\om_k|}{\pi(b+1)R_2}(a u,v)_{\p B_{b_*}^k} &- \e\eta |\p\om_k| \la au \ra_k \la v \ra_k \bigg|
\leqslant  C\e^{\frac{1}{2}}\eta \big(|\ln\eta|^{\frac{1}{2}}+1\big) \|u\|_{W_2^1(\Om^\e)} \|v\|_{W_2^1(\Om^\e)}.
\end{align*}
Two last inequalities and (\ref{3.23}) prove the lemma.
\end{proof}

Lemmata~\ref{lm1.6},~\ref{lm3.3b} imply

\begin{lemma}\label{lm1.3}
Given arbitrary $\b\in\Hinf^1(\g)$, for each $u\in W_2^1(\Om)$, $v\in W_2^1(\Om^\e)$ the estimates
\begin{align*}
&
\|u\|_{W_2^1(\Om)}^2\leqslant C\Big(\fao(u,u)+(\beta u,u)_{L_2(\g)}+\|u\|_{L_2(\Om)}^2\Big),
\nonumber
\\
&
\|v\|_{W_2^1(\Om^\e)}^2\leqslant C\Big(\fae(v,v)+
(av,v)_{L_2(\p\tht_\e^0)}+\|v\|_{L_2(\Om^\e)}^2\Big)
\end{align*}
hold true.
\end{lemma}

We let $\widetilde{\g}:=\{x: \tau=-(b+1)R_2\e,\ s\in \mathds{R}\}$ and by \ref{B1} we see that
$B^k_{b}\cap\widetilde{\g}=\emptyset$
for each $k\in\mathds{M}^\e$ and sufficiently small $\e$.

\begin{lemma}\label{lm7.2}
Let $u\in W_2^2\big(\Om\setminus(\g\cup\widetilde{\g})\big) \cap W_2^1(\Om)$. Then the uniform in $\e$, $k$, and $u$ estimate
\begin{equation*}
\left(\sum\limits_{k\in\mathds{Z}} \|u\|_{C(\overline{B_{R_2\e}(y^k_\e)})}^2
\right)^{\frac{1}{2}}\leqslant C\e^{-\frac{1}{2}}\|u\|_{W_2^2 (\Om\setminus(\g\cup\widetilde{\g}))}
\end{equation*}
holds true.
\end{lemma}

\begin{proof}
We first observe that by standard embedding theorems \cite[Ch. I\!I\!I, Sect. 6, Thm. 3]{Mi} function $u$ is continuous in $\Om_-$
and $\{x:\, -(b+1)R_2\e\leqslant\tau\leqslant 0\}$ and therefore it is continuous in each of balls $B_{R_2\e}(y_k^\e)$.

We denote
$T_k^\pm(\e):=\{x:\, 0<\pm\tau<b R_2\e,\, |s-s_k^\e|<b R_2^\e\}$.
We introduce new variables $\widetilde{\tau}:=\tau\e^{-1}$, $\widetilde{s}:=(s-s_k^\e)\e^{-1}$. Then domain $T_k^\pm(\e)$ is mapped onto $T^\pm:=\{(\widetilde{\tau},\widetilde{s}):\, 0<\pm\widetilde{\tau}<bR_2,\, |\widetilde{s}|<b R_2\}$. Given function $u\in W_2^2(\Om\setminus(\g\cup\widetilde{\g}))$, in the vicinity of curve $\g$ we rewrite it in terms of variables $(\widetilde{\tau},\widetilde{s})$: $v(\widetilde{\tau},\widetilde{s})=u(x)$, and we see that $v\in W_2^2(T^\pm)$. By standard embedding theorems we have the estimate
$\|v\|_{C(\overline{T^\pm})}^2\leqslant C\|v\|_{W_2^1(T^\pm)}^2$.
Then we rewrite this inequality in variables $x$. At that, thanks to the assumptions for curve $\g$, all the coefficients and Jacobians appearing while rewriting derivatives and integrals are bounded uniformly in $\e$, $k$, and $x$. The final estimate is as follows:
\begin{equation*}
\|u\|_{C(\overline{T_k^\pm(\e)})}^2\leqslant C\Big( \|\nabla u\|_{W_2^1(T_k^\pm(\e))}^2+\e^{-2} \|u\|_{L_2(T_k^\pm(\e))}^2\Big).
\end{equation*}
Hence,
\begin{equation}\label{3.23a}
\sum\limits_{k\in\mathds{Z}}\|u\|_{C(\overline{T_k^\pm(\e)})}^2\leqslant C\sum\limits_{k\in\mathds{Z}}\Big( \|\nabla u\|_{W_2^1(T_k^\pm(\e))}^2+\e^{-2} \|u\|_{L_2(T_k^\pm(\e))}^2\Big).
\end{equation}
We integrate inequality (\ref{3.4a}) with $\d=1$ over $\bigcup\limits_{k\in\mathds{Z}} T_k^\pm(\e)$ and obtain:
\begin{equation*}
\sum\limits_{k\in\mathds{Z}} \|u\|_{L_2(T_k^\pm(\e))}^2\leqslant C\e\|u\|_{W_2^1(\Om)}^2.
\end{equation*}
This estimate and (\ref{3.23a}) yield
\begin{equation*}
\sum\limits_{k\in\mathds{Z}} \|u\|_{C(\overline{T_k^\pm(\e)})}^2 \leqslant C\e^{-1} \|u\|_{W_2^2(\Om\setminus(\g\cup\widetilde{g}))}^2.
\end{equation*}
Now the obvious inclusion $B_{R_2\e}(y_k^\e)\subset \overline{T_k^+(\e)\cup T_k^-(\e)}$  completes the proof.
\end{proof}

The next lemma provides apriori estimates for the original and homogenized resolvent.

\begin{lemma}\label{lm1.2}
The estimates
\begin{align}
&\|(\ope-\iu)^{-1}f\|_{W_2^1(\Om^\e)}\leqslant C\|f\|_{L_2(\Om^\e)},\nonumber
\quad \|(\opo{D}-\iu)^{-1}f\|_{W_2^2(\Om_\pm)}\leqslant C\|f\|_{L_2(\Om_\pm)},\nonumber
\\
&\|(\opo{}-\iu)^{-1}f\|_{W_2^2(\Om)}\leqslant C\|f\|_{L_2(\Om)}, \label{1.5}
\\
&
\|(\opo{\beta}-\iu)^{-1}f\|_{W_2^2(\Om\setminus\g)}
\leqslant C(\|\beta\|_{\Hinf^1(\g)}+1)\|f\|_{L_2(\Om)}
\label{1.7}
\end{align}
hold true, where $\beta\in\Hinf^1(\g)$.
\end{lemma}

\begin{proof}
The first estimate is implied by Lemma~\ref{lm1.3} and the
integral identity for $(\ope-\iu)^{-1}f$. And the three other estimates can be proven completely in the same way as Lemma~8.1 in \cite[Ch. I\!I\!I, Sec. 8]{Ld}.
\end{proof}

Given $\beta\in\Hinf^1(\g)$, by $\topo{\b}$ we denote the operator with the differential expression (\ref{2.4}) subject to the boundary conditions
\begin{gather}
[u]_{\widetilde{\g}}=0,\quad \bigg[\frac{\p u}{\p \widetilde{N}^0}\bigg]_{\widetilde{\g}}+\beta u\big|_{\widetilde{\g}}=0,\label{6.2}
\\
\hphantom{\Bigg(}\frac{\p\hphantom{N}}{\p \widetilde{N}^0}:=\sum\limits_{i,j=1}^{2} A_{ij}\nu_i^0 \frac{\p\hphantom{x}}{\p x_j},\quad [u]_{\widetilde{\g}}:=u\big|_{\tau=-(b+1)R_2\e+0}-u\big|_{\tau=-(b+1)R_2\e-0}. \nonumber
\end{gather}
Here the function $\b$ is defined on $\widetilde{\g}$ in the sense that $\b=\b(s)$ at the point $x=\varrho(s)-(b+1)R_2\e\nu^0(s)\in\widetilde{\g}$.
We observe that the normal to $\widetilde{\g}$ coincides with $\nu^0$ and this is why exactly this vector appears in boundary conditions (\ref{6.2}). The associated form is
$\tfoo{\b}(u,v):=\fao(u,v)+(\b u,v)_{L_2(\widetilde{\g})}$
in $L_2(\Om)$ on $\Ho^1(\Om)$.

As $\g$,  curve $\widetilde{\g}$ partitions $\Om$ into two disjoint subdomains $\widetilde{\Om}_\pm$, where $\widetilde{\Om}_+$ is the upper/exterior one. By analogy with \cite[Lem. 2.2]{BorAA}, \cite[Ch. I\!V, Sec. 2.2, 2.3]{Mi}, \cite[Lem. 3.2]{BorIEOP} one can check that
$\Dom(\topo{\b})=\{u\in\Ho^1(\Om):  u\in W_2^2(\widetilde{\Om}_\pm)\text{ and (\ref{6.2}) is satisfied}\}$.

Our last lemma in this section is devoted to estimating the resolvent of operator $\topo{\beta}$.

\begin{lemma}\label{lm6.1}
Let $\beta\in\Hinf^1(\{x: |\tau|<\tau_0/2\})$. Then for any $f\in L_2(\Om)$ and all sufficiently small $\e$ the estimates
\begin{align}
&
\|(\topo{\beta}-\iu)^{-1}f\|_{W_2^2(\Om\setminus\widetilde{\g})}\leqslant C\big(\|\b\|_{\Hinf^1(\g)}+1\big)\|f\|_{L_2(\Om)},
\label{6.3}
\\
&\|(\topo{\beta}-\iu)^{-1}f-(\opo{\beta}-\iu)^{-1} f\|_{W_2^2(\Om\setminus(\g\cup\widetilde{\g}))}\leqslant C\e^{\frac{1}{2}}\|\b\|_{\Hinf^1(\g)}\|f\|_{L_2(\Om)}\label{6.4}
\end{align}
hold true.
\end{lemma}

\begin{proof}
The first estimate can be proven by reproducing the arguments in the proof of Lemma~8.1 in \cite[Ch. I\!I\!I, Sec. 8]{Ld} and keeping track of the dependence on $\beta$. Although now the operator depends on $\e$, the only dependence is in the definition of the curve $\widetilde{\g}$ and its equation depends on $\e$ smoothly. Exactly this fact implies  that estimate (\ref{6.3}) is uniform in $\e$.

We write the integral identities for $u^0:=(\opo{\beta}-\iu)^{-1} f$ and $\tu:=(\topo{\beta}-\iu)^{-1} f$ choosing $\widehat{u}_0:=u^0-\tu$ as the test function. Then we deduct one identity from the other. It yields
\begin{equation}\label{6.5}
\foo{\beta}(\widehat{u}_0,\widehat{u}_0) -\iu\|\widehat{u}_0\|_{L_2(\Om)}^2= (\beta \widetilde{u}_0, \widehat{u}_0)_{L_2(\widetilde{\g})}- (\beta \widetilde{u}_0,\widehat{u}_0)_{L_2(\g)}.
\end{equation}
Since $\g$ has a bounded curvature, we have
$\frac{\displaystyle d\nu^0}{\displaystyle ds}=K\varrho'$,
where $K$ is an uniformly bounded on $\g$ function. Then we can rewrite the right hand side of (\ref{6.5}) as
\begin{equation}
\begin{aligned}
(\beta \widetilde{u}_0,& \widehat{u}_0)_{L_2(\widetilde{\g})}- (\beta \widetilde{u}_0,\widehat{u}_0)_{L_2(\g)}
\\
&=\int\limits_{\mathds{R}}  (\beta\widetilde{u}_0\widehat{u}_0)\big|_{\tau=-(b+1)R_2\e}(1-(b+1)R_2\e K(s))\di s- \int\limits_{\mathds{R}}  (\beta\widetilde{u}_0\widehat{u}_0)\big|_{\tau=0}\di s
\\
&= -(b+1)R_2\e\int\limits_{\mathds{R}}  (\beta\widetilde{u}_0\widehat{u}_0)\big|_{\tau=-(b+1)R_2\e}\di s- \int\limits_{\mathds{R}}  \int\limits_{-(b+1)R_2\e}^{0}
\frac{\p\hphantom{\tau}}{\p\tau} \beta\widetilde{u}_0\widehat{u}_0 \di\tau\di s.
\end{aligned}\label{5.15}
\end{equation}
We then employ (\ref{3.4a}) with $v=\widetilde{u}^0$, $v=\frac{\p \widetilde{u}^0}{\p \tau}$, $v=\widehat{u}^0$, $v=\frac{\p \widehat{u}^0}{\p \tau}$ to obtain
\begin{equation*}
\left|(\beta \widetilde{u}_0, \widehat{u}_0)_{L_2(\widetilde{\g})}- (\beta \widetilde{u}_0,\widehat{u}_0)_{L_2(\g)}\right|\leqslant C\e^{\frac{1}{2}} \|\b\|_{\Hinf^1(\g)} \|\widetilde{u}_0\|_{W_2^1(\Om)}\|\widehat{u}_0\|_{W_2^1(\Om)}.
\end{equation*}
Two last relations, (\ref{6.5}), and Lemma~\ref{lm1.3} yield
\begin{equation*}
\|\widehat{u}^0\|_{W_2^1(\Om)}\leqslant C\e^{\frac{1}{2}}\|\b\|_{\Hinf^1(\g)}\|f\|_{L_2(\Om)}.
\end{equation*}

It remains to estimate $L_2(\Om)$-norm of second derivatives of $\widehat{u}^0$. We again reproduce the arguments in the proof of Lemma~8.1 in \cite[Ch. I\!I\!I, Sec. 8]{Ld}. It leads us to the estimate
\begin{align*}
\|\widehat{u}^0\|_{W_2^2(\Om\setminus(\g\cup\widetilde{\g}))}\leqslant & C(\|\b\|_{\Hinf^1(\{x: |\tau|<\tau_0/2\})}+1)\|\widehat{u}^0\|_{W_2^1(\Om)}
\\
&+ \bigg| \int\limits_{\g} \frac{\p \widehat{u}^0}{\p s} \left(P_1 \frac{\p u^0}{\p s}+P_2u^0\right)\di s-
 \int\limits_{\widetilde{\g}} \frac{\p \widehat{u}^0}{\p s} \left(P_1 \frac{\p u^0}{\p s}+P_2u^0\right)\di s
\bigg|
\end{align*}
where $P_i\in\Hinf^1({x: |\tau|<\tau_0/2})$ are certain functions obeying the inequality
\begin{equation*}
\|P_1\|_{\Hinf^1({x: |\tau|<\tau_0/2})} + \|P_2\|_{\Hinf^1({x: |\tau|<\tau_0/2})}\leqslant C \|\b\|_{\Hinf^1(\{x: |\tau|<\tau_0/2\})},
\end{equation*}
and $C$ are constants independent of $f$, $\beta$, and $\e$. Proceeding as in (\ref{5.15}), we get the desired estimates for $\|\widehat{u}^0\|_{W_2^2(\Om\setminus(\g\cup\widetilde{\g}))}$.
\end{proof}

\section{
Homogenized Dirichlet condition}

In this section we prove Theorem~\ref{th2.1}. Given arbitrary  $f\in L_2(\Om)$, we denote $u^{\e}:=(\ope-\iu)^{-1}f$, $u^0:=(\opo{D}-\iu)^{-1}f$. Estimate (\ref{2.5a}) is equivalent to
\begin{equation}\label{4.0a}
\|u^\e-u^0\|_{W_2^1(\Om^\e)} \leqslant C\e^{\frac{1}{2}} \big(|\ln\eta|^{\frac{1}{2}}+1\big)\|f\|_{L_2(\Om)},
\end{equation}
and in what follows we shall prove exactly this inequality.

Our main idea is to employ the integral identities for $u^\e$ and $u^0$ and to get then a similar identity for $u^\e-u^0$. However,
function $u^\e-u^0$ does not satisfy Dirichlet condition on $\p\tht_0^\e$ and we can not use it as the test function in the integral identity for $u^\e$. To overcome this difficulty, we make use of a boundary corrector.
Namely, let $\chi^{\e}_1(x):=\chi_1\Big(\frac{|\tau|}{R_3\e}\Big)$ as $|\tau|<\tau_0$ and $\chi^{\e}_1(x):=0$ outside the set $\{x: |\tau|<\tau_0\}$,
$\chi_1$ is the cut-off function introduced in the proof of Lemma~\ref{lm3.3b}. We also let $v^\e:=u^{\e}-u^0+\chi_1^\e u^0=(1-\chi_1^\e) u^0$.  Function $v^\e$ vanishes on $\p\tht_0^\e$ and we use it at as the test function in the integral identity for $u^\e$. And our strategy is to estimate independently $W_2^1(\Om^\e)$-norm of $v^\e$ and $\chi_1^\e u^0$. This will lead us estimate (\ref{4.0a}).

Since $v^\e\in\Ho(\Om^\e,\p\Om\cup\p\tht_0^\e)$ and $(1-\chi_1^\e)v^\e\in\Ho(\Om,\p\Om\cup\g)$, we can use these functions as the test ones in the integral identities for operators  $\ope$ and $\opo{D}$:
\begin{equation}\label{3.1}
\begin{aligned}
&\foe(u^{\e},v^{\e})-\iu(u^{\e},v^{\e})_{L_2(\Om^\e)}=(f,v^{\e})_{L_2(\Om^\e)},
\\
&\foo{D}\big(u^0,(1-\chi^{\e}_1) v^{\e}\big) -\iu \big(u^0,(1-\chi^{\e}_1) v^{\e}\big)_{L_2(\Om)}=\big(f,(1-\chi^{\e}_1) v^{\e} \big)_{L_2(\Om)}.
\end{aligned}
\end{equation}
Function $1-\chi^{\e}_1$ vanishes in each $\om_k^\e$ and hence
\begin{equation}\label{3.2}
\begin{gathered}
\big(u^0,(1-\chi^{\e}_1) v^{\e}\big)_{L_2(\Om)}= \big(u^0,(1-\chi^{\e}_1)v^{\e} \big)_{L_2(\Om^\e)},
\\
\big(f,(1-\chi^{\e}_1)v^{\e}\big)_{L_2(\Om)}=\big(f,(1-\chi^{\e}_1) v^{\e}\big)_{L_2(\Om^\e)},
\quad \big(a(1-\chi^{\e}_1)u^0,v^{\e}\big)_{L_2(\p\tht^\e_0)}=0,
\end{gathered}
\end{equation}
and by the definition of $\foo{D}$,
\begin{align}
&\foo{D}\big(u^0,(1-\chi^{\e}_1) v^{\e}\big)= \foo{D} \big((1-\chi^{\e}_1)u^0,v^{\e}\big)+S^\e,\label{3.3}
\\
&S^\e:=- \sum\limits_{i,j=1}^{2} \left(A_{ij}\frac{\p u^0}{\p x_j} \frac{\p \chi^{\e}_1}{\p x_i}, v^{\e}\right)_{L_2(\Om^\e)}
 + \sum\limits_{i,j=1}^{2} \left(A_{ij}u^0\frac{\p \chi^{\e}_1}{\p x_j},\frac{\p v^{\e}}{\p x_i}\right)_{L_2(\Om^\e)}\nonumber
\\
&\hphantom{S^\e:=} +\sum\limits_{j=1}^{2} \left(A_j u^0 \frac{\p\chi^{\e}_1}{\p x_j},v^{\e}\right)_{L_2(\Om^\e)}
- \sum\limits_{j=1}^{2} \left(u^0\frac{\p\chi^{\e}_1}{\p x_j}, A_j v^{\e} \right)_{L_2(\Om^\e)}.\nonumber
\end{align}
We deduct the formulae in (\ref{3.1}) one from the other and employ (\ref{3.2}), (\ref{3.3}),
\begin{equation}\label{3.4}
\foe(v^{\e},v^{\e})-\iu\|v^{\e}\|_{L_2(\Om^\e)}^2= (\chi^{\e}_1f,v^{\e})_{L_2(\Om^\e)}+S^\e.
\end{equation}

Our next step is to estimate the right hand side of the obtained  identity. In order to do it, we need two auxiliary lemmata.

\begin{lemma}\label{lm3.1}
For each $u\in\Dom(\opo{D})$ and $|\tau|<\tau_0/3$ the estimates
\begin{equation*}
|u(s,\tau)|^2\leqslant C\tau^2 \|u(s,\cdot)\|_{W_2^2\left(-\frac{\tau_0}{2},\frac{\tau_0}{2}\right)}^2,
\quad
|\nabla_{s,\tau} u(s,\tau)|^2\leqslant C\|\nabla_{s,\tau} u(s,\cdot)\|_{W_2^1\left(-\frac{\tau_0}{2},\frac{\tau_0}{2}\right)}^2\
\end{equation*}
hold true. 
\end{lemma}

\begin{proof}
The desired estimates follow from the obvious relations
\begin{equation*}
|u(s,\tau)|^2=\bigg|\int\limits_{0}^{\tau} \frac{\p u}{\p\tau} (s,t)\di t\bigg|^2\leqslant |\tau| \int\limits_{-\frac{\tau_0}{2}}^{\frac{\tau_0}{2}} \left|\frac{\p u}{\p\tau}(s,t)\right|^2\di t
\end{equation*}
and (\ref{3.4a}) with $v=\frac{\p u}{\p \tau}$.
\end{proof}

\begin{lemma}\label{lm3.4}
The estimate
\begin{equation*}
\|v^{\e}\|_{L_2(\Pi^{R_3\e}\setminus\tht^\e)}\leqslant C\e\big(|\ln\eta(\e)|^{\frac{1}{2}}+1\big)\|\nabla v^{\e}\|_{L_2(\Om^\e)}
\end{equation*}
holds true.
\end{lemma}

\begin{proof}  We extend the function $v^{\e}$ by zero inside $\tht_0^\e$. Since $v^{\e}$ vanishes on $\p\tht^0_\e$,  the  extension belongs to $W_2^1(\Om\setminus\tht_1^\e)$ and has the same $L_2$- and $W_2^1$-norm.

By assumption \ref{B1}, the ball $B_{R_1\e\eta}\big(x^k+y_k^\e\big)$ lies inside $\om_k^\e$ for each $k\in\mathds{M}_0^\e$. We introduce polar coordinates $(r,\tht)$ centered at $x^k+y_k^\e$ and associated with variables $(s,\tau)$.  Since $v^{\e}=0$ inside $B_{R_1\e\eta}\big(x^k+y_k^\e\big)$, $k\in\mathds{M}_0^\e$, we have
\begin{equation}\label{4.5a}
v^{\e}(x)=\int\limits_{R_1\e\eta}^{r} \frac{\p v^{\e}}{\p r}\di r,\quad |v^{\e}(x)|^2\leqslant \ln\frac{r}{R_4\e\eta} \int\limits_{R_5\e\eta}^{r} \left|\frac{\p v^{\e}}{\p r}\right|^2 r\di r
\end{equation}
for some $R_4>0$.

It follows from Assumption \ref{A5} that the domain $\Pi^{R_3\e}\setminus \bigcup\limits_{k\in\mathds{M}_1^\e} B^k_{b_*}$ can be covered by the union of star-shaped domains so that each of these domains contains exactly one of the balls $B_{R_1\e\eta}(x_k+y_k^\e)$ and is contained in the ball $B_{\widetilde{R_3}\e}(y_k^\e)$, where $\widetilde{R_3}$ is a fixed constant. Integrating then (\ref{4.5a}) over these star-shaped domains, we arrive at the estimates
\begin{align}\label{3.9b}
&\|v^\e\|_{L_2(\Pi^{R_3\e}\setminus \bigcup\limits_{k\in\mathds{M}_1^\e}B^k_{b_*})}
  \leqslant C\e^2(|\ln\eta|+1)  \|\nabla v\|_{L_2(\Pi^{R_3\e}\setminus \bigcup\limits_{k\in\mathds{M}_1^\e} B^k_{1})}^2,
\\
&\sum\limits_{k\in\mathds{M}_1^\e} \|v^\e\|_{L_2(\p B^k_{b_*})}^2 \leqslant C\e\eta \|v\|_{W_2^1(\Pi^{R_3\e}\setminus \bigcup\limits_{k\in\mathds{M}_1^\e} B^k_{b})}^2.\nonumber
\end{align}
The latter estimate and Lemma~\ref{lm3.3b} yield
\begin{equation*}
\sum\limits_{k\in\mathds{M}_1^\e} \|v\|_{L_2(\widehat{B}^k_{b_*})}^2 \leqslant C\e^2\eta^2 \|v\|_{W_2^1(\Pi^{R_3\e}\setminus\tht^\e)}.
\end{equation*}
Combining this estimate and (\ref{3.9b}), we complete the proof.
\end{proof}

Let us estimate the right hand side of (\ref{3.4}). By Lemma~\ref{lm3.4} we get
\begin{align*}
\big|\big(\chi^{\e}_1 f,v^{\e}\big)_{L_2(\Om^\e)} \big|= & \big|\big(f, \chi^{\e}_1 v^{\e}\big)_{L_2(\Pi^{R_3\e}\setminus\tht^\e)} \big| \leqslant C\|f\|_{L_2(\Om)}\|v^{\e}\|_{L_2(\Pi^{R_3\e}\setminus\tht^\e)}
\\
\leqslant & C\e\big(|\ln\eta|^{\frac{1}{2}}+1\big)\|f\|_{L_2(\Om)} \|\nabla v^{\e}\|_{L_2(\Om^\e)}.
\end{align*}
In the same way, employing Lemmata~\ref{lm1.3},~\ref{lm1.2},~\ref{lm3.1},~\ref{lm3.4}, we obtain the estimate for the first term in $S^\e$:
\begin{align*}
&\left|\sum\limits_{i,j=1}^{2} \left(A_{ij}\frac{\p u^0}{\p x_j} \frac{\p \chi^{\e}_1}{\p x_i}, v^{\e}\right)_{L_2(\Om^\e)}\right|\leqslant C\e^{-1} \|\nabla u^0\|_{L_2(\Pi^{R_3\e})} \|v^{\e}\|_{L_2(\Pi^{R_3\e}\setminus\tht^\e)}
\\
&\hphantom{\sum A}\leqslant  C\e^{\frac{1}{2}} \big(|\ln\eta|^{\frac{1}{2}}+1\big) \|u^0\|_{W_2^2(\Om)} \|\nabla v^{\e}\|_{L_2(\Om^\e)}
\leqslant  
C\e^{\frac{1}{2}} \big(|\ln\eta|^{\frac{1}{2}}+1\big) \|f\|_{L_2(\Om)} \|v^{\e}\|_{W_2^1(\Om^\e)}.
\end{align*}
The other terms in $S^\e$ are estimated in the same way,
\begin{align*}
&\left|\sum\limits_{i,j=1}^{2} \left(A_{ij}u^0\frac{\p \chi^{\e}_1}{\p x_j},\frac{\p v^{\e}}{\p x_i}\right)_{L_2(\Om^\e)} \right|\leqslant C\e^{-1}\|u^0\|_{L_2(\Pi^{R_3\e})}\|v^{\e}\|_{W_2^1(\Om^\e)}
\leqslant C \e^{\frac{1}{2}} \|f\|_{L_2(\Om)} \|v^{\e}\|_{W_2^1(\Om^\e)},
\\
&\left|
\sum\limits_{j=1}^{2} \left(A_j u^0 \frac{\p\chi^{\e}_1}{\p x_j},v^{\e}\right)_{L_2(\Om^\e)}
-\sum\limits_{j=1}^{2} \left(u^0, A_j v^{\e} \frac{\p\chi^{\e}_1}{\p x_j}\right)_{L_2(\Om^\e)}\right|
\\
&\hphantom{\sum\limits_{j=1}^{2}\bigg(}
\leqslant C\e^{-1} \|u^0\|_{L_2(\Om)} \|v^{\e}\|_{L_2(\Pi^{R_3\e}\setminus\tht^\e)}
\leqslant C\e^{\frac{3}{2}} \big(|\ln\eta|^{\frac{1}{2}}+1\big)  \|f\|_{L_2(\Om)} \|v^{\e}\|_{W_2^1(\Om^\e)}.
\end{align*}
Last four inequalities imply the estimate for the right hand side of (\ref{3.4}):
\begin{equation*}
\Big|(\chi^{\e}_1f,v^{\e})_{L_2(\Om^\e)}+S^\e\Big| \leqslant C\e^{\frac{1}{2}} \big(|\ln\eta|^{\frac{1}{2}}+1\big) \|f\|_{L_2(\Om)} \|v^\e\|_{W_2^1(\Om^\e)}.
\end{equation*}
Now we apply Lemma~\ref{lm1.3} and arrive at the inequality
\begin{equation}\label{4.5}
\|v^{\e}\|_{W_2^1(\Om^\e)}\leqslant C\e^{\frac{1}{2}} \big(|\ln\eta(\e)|^{\frac{1}{2}}+1\big)  \|f\|_{L_2(\Om)}.
\end{equation}

It remains to estimate the norm $\|\chi^{\e}_1 u^0\|_{W_2^1(\Om^\e)}$ to complete the proof. Employing Lemmata~\ref{lm1.2},~\ref{lm3.1}, one can check easily that
\begin{align*}
&\|\chi^{\e}_1 u^0\|_{L_2(\Om^\e)}\leqslant C\e^{\frac{3}{2}}\|f\|_{L_2(\Om)},
\\
& \|\nabla \chi^{\e}_1 u^0\|_{L_2(\Om^\e)} \leqslant C\left(
 \|\chi^{\e}_1 \nabla u^0\|_{L_2(\Om^\e)} + \e^{-1}\|u^0\|_{L_2(\Om^\e)}\right) \leqslant C\e^{\frac{1}{2}} \|f\|_{L_2(\Om)}.
\end{align*}
These inequalities and (\ref{4.5}) imply (\ref{4.0a}) that completes the proof.

\section{Robin condition} 

In this section we prove Theorems~\ref{th2.2},~\ref{th2.3}. We begin with Theorem~\ref{th2.2}.

\subsection{Proof of Theorem~\ref{th2.2}}

Let $f\in L_2(\Om)$, $u^\e:=(\ope-\iu)^{-1}f$, $u^0:=(\opo{}-\iu)^{-1}f$, $v^\e:=u^\e-u^0$.
The desired estimates for the resolvents are equivalent to
\begin{align}
&\|u^\e-u^0\|_{W_2^1(\Om^\e)}\leqslant C \eta(\e)\big(|\ln\eta|^{\frac{1}{2}}+1\big)\|f\|_{L_2(\Om)}, && a\not\equiv 0,\quad \eta\to+0,\label{5.0a}
\\
&
\|u^\e-u^0\|_{W_2^1(\Om^\e)}\leqslant C \e^{\frac{1}{2}}\eta(|\ln\eta|^{\frac{1}{2}}+1)\|f\|_{L_2(\Om)}, && a\equiv 0.\label{5.0b}
\end{align}
In what follows we prove the above estimates.

By the assumption $\mathds{M}_0^\e=\emptyset$ we have $\tht^\e_0=\emptyset$, $\tht^\e_1=\tht^\e$. Since $u^0\in W_2^2(\Om)$, by the standard embedding theorems the function $\left(\frac{\p\hphantom{N}}{\p N^\e}+a\right)u^0$ belongs to $L_2(\p\tht^\e)$. Then function $v^\e$ is the generalized solution to the boundary value problem
\begin{align*}
&\left(- \sum\limits_{i,j=1}^{2} \frac{\p\hphantom{x}}{\p x_i} A_{ij} \frac{\p\hphantom{x}}{\p x_j} + \sum\limits_{j=1}^{2} A_j\frac{\p\hphantom{x}}{\p x_j}-\frac{\p\hphantom{x}}{\p x_j}\overline{A_j} + A_0- \iu\right)v^\e=0\quad\text{in}\quad \Om^\e,
\\
&\hphantom{\Bigg(-}v^\e=0\quad \text{on}\quad \p\Om,\qquad \left(\frac{\p\hphantom{N}}{\p N^\e}+a\right)v^\e=-\left(\frac{\p\hphantom{N}}{\p N^\e}+a\right)u^0\quad \text{on}\quad \p\tht^\e.
\end{align*}
Taking $v^\e$ as the test function, we write the associated integral identity
\begin{equation}\label{5.1}
\foe(v^\e,v^\e)-\iu\|v^\e\|_{L_2(\Om^\e)}^2=-\left(\left(\frac{\p\hphantom{N}} {\p N^\e}+a\right)u^0,v^\e\right)_{L_2(\p\tht^\e)}.
\end{equation}
The main idea of our proof is to estimate the right hand side of this identity and to get then the desired estimate for $v^\e$.

Assume $\eta$ is arbitrary, not necessary small. It is clear that
\begin{equation}\label{5.2}
\left|\left(\left(\frac{\p\hphantom{N}} {\p N^\e}+a\right)u^0,v^\e\right)_{L_2(\p\tht^\e)}\right|\leqslant \left|\left( \frac{\p u^0} {\p N^\e} ,v^\e\right)_{L_2(\p\tht^\e)}\right|+C_*\|u^0\|_{L_2(\p\tht^\e)}\|v^\e\|_{L_2(\p\tht^\e)}.
\end{equation}
If $a\equiv0$,   constant $C_*$  vanishes.

Let us estimate the  term $\left( \frac{\p u^0} {\p N^\e} ,v^\e\right)_{L_2(\p\tht^\e)}$.  We first integrate by parts:
\begin{align*}
&\int\limits_{\widehat{B}^k_{1}} \overline{v^\e} \left(\sum\limits_{i,j=1}^{2} \frac{\p\hphantom{x}}{\p x_i} A_{ij} \frac{\p\hphantom{x}}{\p x_j} +\sum\limits_{j=1}^{2}\frac{\p\hphantom{x}}{\p x_j}\overline{A_j}\right)u^0\di x
=\left(\frac{\p u^0}{\p N^\e},v^\e\right)_{L_2(\tht^\e)}
\\
&- \left(\frac{\p u^0}{\p N^\e_*},v^\e\right)_{L_2(\p B^k_{1})} -\sum\limits_{i,j=1}^{2} \left(A_{ij} \frac{\p u^0}{\p x_j}, \frac{\p v^\e}{\p x_i}\right)_{L_2(\widehat{B}^k_{1})} + \sum\limits_{j=1}^{2}\left(u^0, A_j \frac{\p v^\e}{\p x_j}\right)_{L_2(\widehat{B}^k_{1})},
\end{align*}
where $\frac{\p\hphantom{N}}{\p N_*^\e}$ is introduced in the same way as $\frac{\p\hphantom{N}}{\p N^\e}$, but instead of $\p\om^\e_k$ we take $\p B^k_{1}$. Hence,
\begin{equation}
\begin{aligned}
&\left(\frac{\p u^0}{\p N^\e},v^\e\right)_{L_2(\tht^\e)}=
\int\limits_{\widehat{B}^k_{1}} \overline{v^\e} \left(\sum\limits_{i,j=1}^{2} \frac{\p\hphantom{x}}{\p x_i} A_{ij} \frac{\p\hphantom{x}}{\p x_j}
+\sum\limits_{j=1}^{2}\frac{\p\hphantom{x}}{\p x_j}\overline{A_j}\right)u^0\di x
\\
&+ \left(\frac{\p u^0}{\p N^\e_*},v^\e\right)_{L_2(\p B^k_{1})} +\sum\limits_{i,j=1}^{2} \left(A_{ij} \frac{\p u^0}{\p x_j}, \frac{\p v^\e}{\p x_i}\right)_{L_2(\widehat{B}^k_{1})} - \sum\limits_{j=1}^{2}\left(u^0, A_j \frac{\p v^\e}{\p x_j}\right)_{L_2(\widehat{B}^k_{1})}.
\end{aligned}\label{5.14}
\end{equation}

We consider the boundary value problem
\begin{equation*}
\D U_{k,i}^\e=0\quad\text{in}\quad B^k_{b_*}\setminus B^k_{1},
\quad
\frac{\p U_{k,i}^\e}{\p r}=\frac{\z^k_i}{r}\quad\text{on}\quad \p B^k_{1},\qquad \frac{\p U_{k,i}^\e}{\p r}=0\quad\text{on}\quad \p B^k_{b_*},
\end{equation*}
where, we remind, $\z^k=(\z^k_1, \z^k_2):=x-y_k^\e$, $r=|\z^k|$. It has the explicit solution
\begin{equation*}
U_{k,i}^\e(x):=\frac{4}{(b+1)^2-4} \left(\z^k_i + \frac{(b+1)^2 R_2^2}{4}\e^2\eta^2 \frac{\z^k_i}{r^2}  \right),
\end{equation*}
satisfying the uniform pointwise estimate
\begin{equation}\label{5.10}
|\nabla U_{k,i}^\e|\leqslant C\quad\text{in}\quad B^k_{b_*}\setminus B^k_{1}.
\end{equation}
Then integrating  by parts in the identity
\begin{equation*}
0=\int\limits_{B^k_{b_*}\setminus B^k_{1}} \left(\sum\limits_{i,j=1}^{2} A_{ij} \frac{\p u^0}{\p x_j} \overline{v^\e} \D U_{k,i}^\e+ \sum\limits_{j=1}^2\overline{A_j} u^0 \overline{v^\e}  \D U_{k,j}^\e\right)\di x,
\end{equation*}
we get
\begin{equation*}
\left(\frac{\p u^0}{\p N^\e_*},v^\e\right)_{L_2(\p B^k_{1})} = -\int\limits_{B^k_{b_*}\setminus B^k_{1}} \left(\sum\limits_{i,j=1}^{2} \nabla U_{k,i}^\e \cdot\nabla A_{ij} \frac{\p u^0}{\p x_j} \overline{v^\e} + \sum\limits_{j=1}^{2} \nabla U_{k,j}^\e \cdot\nabla \overline{A_j} u^0 \overline{v^\e} \right)\di x.
\end{equation*}
By (\ref{5.14}), (\ref{5.10}), (\ref{1.5}), and (\ref{3.11}) with $v=v^\e$, $v=u^0$, $v=\frac{\p u^0}{\p x_i}$ it follows that
\begin{equation}
\begin{aligned}
\left|\left(\frac{\p u^0}{\p N^\e}, v^\e\right)_{L_2(\p\tht^\e)} \right|
 \leqslant & C \sum\limits_{k\in\mathds{M}^\e}\bigg( \|u^0\|_{W_2^1(\widehat{B}^k_{b_*})} \|\nabla v^\e\|_{L_2(\widehat{B}^k_{b_*})}
+ \|u^0\|_{W_2^2(\widehat{B}^k_{b_*})} \|v^\e\|_{L_2(\widehat{B}^k_{b_*})}\bigg)
\\
 \leqslant & C\e^{\frac{1}{2}}\eta (|\ln\eta|^{\frac{1}{2}}+1)\|f\|_{L_2(\Om^\e)} \|v^\e\|_{W_2^1(\Om^\e)}.
\end{aligned}\label{5.12}
\end{equation}
If $a\equiv0$,  we substitute the obtained identity into (\ref{5.2}) and since $C_*=0$, by identity (\ref{5.1}) and estimate (\ref{1.5}) we arrive at (\ref{5.0a}). It proves the theorem for the case $a\equiv0$.

If $a\not\equiv0$, inequality (\ref{3.12}) for $v^\e$, estimate (\ref{1.5}) and Lemma~\ref{lm7.2} for $u^0$ imply the estimate for the last term in the right hand side of (\ref{5.2}),
\begin{align*}
\|u^0\|_{L_2(\p\tht^\e)}\|v^\e\|_{L_2(\p\tht^\e)}\leqslant & C\eta(|\ln\eta|^{\frac{1}{2}}+1)\|u^0\|_{W_2^2(\Om^\e)}\|v^\e\|_{W_2^1(\Om^\e)} \\
\leqslant& C \eta(|\ln\eta|^{\frac{1}{2}}+1)\|f\|_{L_2(\Om)}\|v^\e\|_{W_2^1(\Om^\e)}.
\end{align*}
In the same way how (\ref{5.0a}) was obtained, the last estimate and (\ref{5.12}) follow (\ref{5.0b}) that proves the theorem for the case $a\not\equiv0$.

\subsection{Proof of Theorem~\ref{th2.3}}

Given $f\in L_2(\Om)$, we let $u^\e:=(\ope-\iu)^{-1}f$,  $u^0:=(\opo{\a  a}-\iu)^{-1}f$.  We need to prove the estimate
\begin{equation}\label{5.6a}
\|u^\e-u^0\|_{W_2^1(\Om^\e)}\leqslant C(\e^{\frac{1}{2}}+\vk)\|f\|_{L_2(\Om)}.
\end{equation}
At the same time, 
curve $\g$ can cross the holes while the functions in the domain of homogenized operator $\opo{\a a}$ have a jump of the normal derivative at this curve. It causes troubles in getting integral identity for $u^\e-u^0$ and in further estimating. This is why we consider curve $\widetilde{\g}$ and operator $\topo{\a^0 a}$ introduced in Section~2, see (\ref{6.2}). Curve $\widetilde{\g}$ does not intersect the holes and this fact allows us to get an estimate similar to (\ref{5.6a})  for $v^\e:=u^\e-\tu$, where $\tu:=(\topo{\a a}-\iu)^{-1}f$. After that we estimate the function $u^0-\tu$ by Lemma~\ref{lm6.1} and it gives (\ref{5.6a}).

Function $v^\e$ is a generalized solution  to the boundary value problem
\begin{align*}
&\left(- \sum\limits_{i,j=1}^{2} \frac{\p\hphantom{x}}{\p x_i} A_{ij} \frac{\p\hphantom{x}}{\p x_j} + \sum\limits_{j=1}^{2} A_j\frac{\p\hphantom{x}}{\p x_j}-\frac{\p\hphantom{x}}{\p x_j}\overline{A_j} + A_0- \iu\right)v^\e=0\quad\text{in}\quad \Om^\e,\qquad v^\e=0\quad \text{on}\quad \p\Om,
\\
&
\left(\frac{\p\hphantom{N}}{\p N^\e}+a\right)v^\e=-\left(\frac{\p\hphantom{N}}{\p N^\e}+a\right)\widetilde{u}^0\quad \text{on}\quad \p\tht^\e,
\qquad
[v^\e]_{\widetilde{\g}}=0,\quad \bigg[\frac{\p v^\e}{\p \widetilde{N}^0}\bigg]_{\widetilde{\g}}+\a  a \tu\big|_{\widetilde{\g}}=0.
\end{align*}
We write the associated integral identity with $v^\e$ as the test function,
\begin{equation}\label{6.7}
\foe(v^\e,v^\e)-\iu\|v^\e\|_{L_2(\Om^\e)}^2=
-\left( \frac{\p\tu} {\p N^\e},v^\e\right)_{L_2(\p\tht^\e)}-(a\tu,v^\e)_{L_2(\p\tht^\e)}
+(\alpha a\tu,v^\e)_{L_2(\widetilde{\g})}.
\end{equation}
Let us estimate the right hand side of this identity.

Proceeding as in (\ref{5.14}), (\ref{5.10}), (\ref{5.12})   and employing (\ref{6.3}) instead of (\ref{1.7}), we obtain
\begin{equation}\label{6.8}
\left|\left(\frac{\p \tu}{\p N^\e},v^\e\right)_{L_2(\tht^\e)}\right|\leqslant
C\e^{\frac{1}{2}}\eta(|\ln\eta|^{\frac{1}{2}}+1) \|f\|_{L_2(\Om)}\|v^\e\|_{W_2^1(\Om^\e)}.
\end{equation}

Let $\xi=(\xi_1,\xi_2)$ be Cartesian coordinates in $\mathds{R}^2$,  $\Xi:=\{\xi: |\xi_1|<b R_2,\ |\xi_2|<(b+1) R_2\}\subset \mathds{R}^2$.  We consider the Neumann boundary value problem
\begin{align*}
&\D Y=0\quad\text{in}\quad \Xi\setminus B_{b_*R_2\eta(0)},\qquad \frac{\p Y}{\p|\xi|}=1\quad\text{on}\quad \p B_{b_*R_2
\eta}(0),
\\
&\frac{\p Y}{\p\nu}=\frac{\pi (b+1)\eta}{2b} \quad\text{on}\quad \{\xi: |\xi_1|<b R_2,\ \xi_2=-(b+1) R_2\},
\\
&\frac{\p Y}{\p\nu}=0\hphantom{\pi(b+.)\eta} \quad\text{on}\quad \p\Xi\setminus\{\xi: |\xi_1|<b R_2,\ \xi_2=(b+1) R_2\},
\end{align*}
where $\nu$ is the outward normal to $\p\Xi$. This problem satisfies the solvability condition
\begin{equation*}
\int\limits_{\p B_{b_* R_2\eta}(0)} \frac{\p Y}{\p |\xi|}\di s=\int\limits_{ \{\xi: |\xi_1|<b R_2,\ \xi_2=-(b+1) R_2\}
} \frac{\p Y}{\p\nu}\di s.
\end{equation*}
There exists the unique generalized solution satisfying the identity
\begin{equation*}
\int\limits_{\Xi\setminus B_{b_*R_2\eta(0)}} Y \di \xi=0.
\end{equation*}
This solution belongs to  $\Hinf^1(\Xi\setminus B_{b_*R_2\eta}(0))$, see \cite[Ch. I\!I\!I, Sect. 12]{Ld}.

In a vicinity of each point $y_k^\e$  we introduce rescaled variables  $\xi^k=(\xi_1^k,\xi_2^k)$ by the rule
$x=\e 
\big(\xi_1^k\varrho'(s_k^\e)-\xi_2^k\nu^0(s_k^\e)
\big)+y_k^\e$. The axes $\xi_2^k=0$ and $\xi_1^k=0$ are directed along the tangential and normal vectors $\varrho'(s_k^\e)$ and $\nu^0(s_k^\e)$ to the curve $\g$ at the point $s_k^\e$, and the point $\xi^k=0$ is located at $y_k^\e$. We define $\Xi_k^\e:=\{x: \xi^k\in\Xi\}$, $Y_k^\e(x):=Y(\xi^k)$.

We make an integration by parts similar to (\ref{3.10b}):
\begin{align*}
0=\e
\int\limits_{\Xi_k^\e\setminus B_{b_*}^k} a \widetilde{u}^0 \overline{v^\e} \D Y_k^\e \di x= & \frac{\pi(b+1)\eta}{2 b} (a \widetilde{u}^0, v^\e)_{L_2(\Ups_k^\e)}-(a \widetilde{u}^0, v^\e)_{L_2(\p B^k_{b_*})}
\\
&-\e
 \int\limits_{\Xi_k^\e\setminus B_{b_*}^k} \nabla a \widetilde{u}^0 \overline{v^\e} \cdot \nabla Y_k^\e \di x,
\end{align*}
where $\Ups_k^\e:=\{x: |\xi_1^k|<b R_2, \xi_2^k=-(b+1)R_2\}$.
Employing this identity, Lemma~\ref{lm3.7}, (\ref{6.3}) and (\ref{3.11}) with $v=v^\e$, $v=\widetilde{u}^0$, $v=\frac{\p\widetilde{u}^0}{\p x_i}$, as in (\ref{5.12}) we obtain
\begin{equation}\label{5.16}
\bigg|(a \widetilde{u}^0,v^\e)_{L_2(\p\tht^\e)} -\sum\limits_{k\in\mathds{M}^\e} \frac{|\p\om_k|\eta}{2 b R_2} (a \widetilde{u}^0,v^\e)_{L_2(\Ups_k^\e)}\bigg| \leqslant C\e^{\frac{1}{2}}\|f\|_{L_2(\Om^\e)} \|v^\e\|_{W_2^1(\Om^\e)}.
\end{equation}

For each $|\xi_1^k|<b R_2$, by $\tau_k^\e(\xi_1^k)$ we denote the value of the variable $\tau$ corresponding to the point $x=\e\xi_1^k\varrho'(s_k^\e)-(b+1) R_2\e \nu^0(s_k^\e)$. It is easy to check that
\begin{equation}\label{6.12}
|\tau_j^\e(\xi_1^j)-(b+1) R_2\e\eta|\leqslant C\e^2
\end{equation}
uniformly in $k\in\mathds{M}^\e$,  sufficiently small $\e$, and $|\xi_1^k|<b R_2$. We also observe that the integration over $\widetilde{\g}$ can be expressed as the integration w.r.t. $s\in\mathds{R}$ with the differential
$(1-(b+1)R_2\e K(s))\di s$,
where  function $K$ was introduced in the proof of Lemma~\ref{lm6.1}. Integrating by parts, we have
\begin{align*}
\sum\limits_{k\in\mathds{M}^\e} \frac{|\p\om_k|\eta}{2 b R_2} (a\tu,v^\e)_{L_2(\Ups_j^\e)}= &\sum\limits_{k\in\mathds{M}^\e} \frac{|\p\om_k|\eta}{2 b R_2} \int\limits_{s_k^\e-b R_2\e \eta}^{s_k^\e + b R_2\e \eta} (a\tu \overline{v^\e})\big|_{\tau=-(b+1)R_2\e}\di s
\\
&+ \sum\limits_{k\in\mathds{M}^\e} \frac{|\p\om_k|\eta}{2 b R_2} \int\limits_{s_k^\e-b R_2\e}^{s_k^\e + b R_2\e}
\di s \int\limits_{-(b+1)R_2\e}^{\tau_k^\e\left(\frac{s-s_k^\e}{\e}\right)}
\frac{\p\hphantom{\tau}}{\p\tau} (a\tu \overline{v^\e})\di\tau.
\end{align*}
Now we employ the definition of function $\a_\e$, estimates (\ref{3.21}) with $v=v^\e$ and $v=\widetilde{u}^0$, (\ref{3.15}), (\ref{6.3}), (\ref{5.16}), (\ref{6.12}) to obtain
\begin{equation}\label{6.1}
\bigg|(a \widetilde{u}^0,v^\e)_{L_2(\p\tht^\e)} - (\a_\e  a \widetilde{u}^0,v^\e)_{L_2(\widetilde{\g})}\bigg|\leqslant C\e^{\frac{1}{2}}\|f\|_{L_2(\Om^\e)} \|v^\e\|_{W_2^1(\Om^\e)}.
\end{equation}

Our final step is

\begin{lemma}\label{lm6.3}
Function $\alpha_\e$ is bounded uniformly in $\e$ in the norm of space $L_\infty(\g)$. The estimate
\begin{equation*}
\big|\big((\a^\e-\a)a\widetilde{u}^0,v^\e\big)_{L_2(\widetilde{\g})}\big|\leqslant C\vk(\e)\|f\|_{L_2(\Om)}\|v^\e\|_{W_2^1(\Om^\e)}
\end{equation*}
holds true.
\end{lemma}
\begin{proof}
The boundedness follows directly from the definition of function $\a^\e$. We shall prove the desired estimate only in the case of an infinite curve, since for a finite curve the proof is completely the same.

Denote $\widetilde{\g}_k^\e:=\{x:\ \tau=-\e(b+1)R_2,\ k<s<k+1\}$.  Since $\widetilde{u}^0\in W_2^2(\Om\setminus\widetilde{\g})$, $v^\e\in W_2^1(\Om^\e)$, the traces of these functions on $\widetilde{\g}_k^\e$ belong respectively to $W_2^{\frac{3}{2}}(\widetilde{\g}_k^\e)$ and $W_2^{\frac{1}{2}}(\widetilde{\g}_k^\e)$. We expand then these traces into Fourier series
\begin{align*}
&\widetilde{u}^0=\sum\limits_{p\in\mathds{Z}} c^{u,k}_p \E^{2\iu\pi p(s-k)},\quad \big(1-\e(b+1)K\big)a v^\e=\sum\limits_{p\in\mathds{Z}} c^{v,k}_p \E^{2\iu\pi p(s-k)},
\\
&\a^\e-\a=\sum\limits_{p\in\mathds{Z}} c^{\a,k}_p \E^{2\iu\pi p(s-k)}\quad\text{on}\quad \widetilde{\g}_k^\e,
\end{align*}
and we have the uniform estimates
\begin{align*}
&\sum\limits_{p\in\mathds{Z}} |c^{u,k}_p|^2 (|p|+1)^3 \leqslant C\|\widetilde{u}^0\|_{W_2^{\frac{3}{2}}(\widetilde{\g}_k^\e)}^2 \leqslant C \|\widetilde{u}^0\|_{W_2^2(\Pi_k^\e)}^2,
\\
&\sum\limits_{p\in\mathds{Z}} |c^{v,k}_p|^2 (|p|+1) \leqslant C\|\big(1-\e(b+1)K\big) a v^\e\|_{W_2^{\frac{1}{2}}(\widetilde{\g}_k^\e)}^2 \leqslant C \|v^\e\|_{W_2^1(\Pi_k^\e)}^2,
\end{align*}
where $\Pi_k^\e:=\{x: -\tau_0/2<\tau<-\e(b+1)R_2,\ k<s<k+1\}$. We employ the above expansions and estimates together with the Cauchy-Schwarz inequality  as follows,
\begin{equation}\label{6.6}
\begin{aligned}
\big|\big((\a^\e&-\a)a\widetilde{u}^0,v^\e\big)_{L_2(\widetilde{\g})}\big|\leqslant \sum\limits_{k\in\mathds{Z}} \big|\big((\a^\e-\a)a\widetilde{u}^0,v^\e\big)_{L_2(\widetilde{\g}_k^\e)}\big|=
\sum\limits_{k\in\mathds{Z}}\left|\sum\limits_{p,q\in\mathds{Z}} c^{u,k}_p \overline{c^{v,k}_q} c^\a_{q-p}\right|
\\
&\leqslant \sum\limits_{k\in\mathds{Z}} \left(\sum\limits_{p,q\in\mathds{Z}} |c^{u,k}_p|^2 |c^{v,k}_q|^2 (1+|p|)^3(1+|q|)\right)^{\frac{1}{2}} \left(\sum\limits_{p,q\in\mathds{Z}}  \frac{|c^{\a,k}_{q-p}|^2}{(1+|p|)^3(1+|q|)} \right)^{\frac{1}{2}}
\\
&\leqslant C \|\widetilde{u}^0\|_{W_2^2(\Om\setminus\widetilde{\g})} \|v^\e\|_{W_2^1(\Om^\e)} \sup\limits_{k\in\mathds{Z}}\left(\sum\limits_{p,q\in\mathds{Z}}  \frac{|c^{\a,k}_{q}|^2}{(1+|p|)^3(1+|p+q|)} \right)^{\frac{1}{2}}.
\end{aligned}
\end{equation}
Let us estimate the supremum in the last inequality by $C\vk$. Indeed, for each $k\in \mathds{Z}$,
\begin{equation*}
\sum\limits_{p,q\in\mathds{Z}}  \frac{|c^{\a,k}_{q}|^2}{(1+|p|)^3(1+|p+q|)}\leqslant C \sum\limits_{q\in\mathds{Z}} |c^{\a,k}_{q}|^2 \int\limits_{\mathds{R}} \frac{\di t}{(1+|t|)^3(1+|t+q|)}
\leqslant C \sum\limits_{q\in\mathds{Z}} \frac{|c^{\a,k}_{q}|^2}{|q|+1}.
\end{equation*}
Now it remains to employ (\ref{2.17}) and (\ref{6.6}) to complete the proof.
\end{proof}

The proven lemma and (\ref{6.8}),  (\ref{6.1}) yield that the right hand side of (\ref{6.7}) is estimated by $C(\e^{\frac{1}{2}}+\vk)\|f\|_{L_2(\Om)}\|v^\e\|_{W_2^1(\Om^\e)}$. By Lemma~\ref{lm1.3} it leads us to the estimate
\begin{equation*}
\|v^\e\|_{W_2^1(\Om^\e)}\leqslant C(\e^{\frac{1}{2}}+\vk)\|f\|_{L_2(\Om)}.
\end{equation*}
This estimate, the definition of $v^\e$, and (\ref{6.4}) imply (\ref{5.6a}) that completes the proof.

\section{
Homogenized delta-interaction for Dirichlet condition}

This section is devoted to the proof of Theorem~\ref{th2.4}.
Since here homogenized operator $\opo{\b}$ involves boundary condition (\ref{2.13}), as in the proof of Theorem~\ref{th2.3}, we introduce operator $\topo{\b}$ with $\b$ defined in the statement of the theorem.  Given $f\in L_2(\Om)$, we let $u^\e:=(\ope-\iu)^{-1}f$,   $\tu:=(\topo{\beta}-\iu)^{-1}f$, $v^\e:=u^\e-\tu$.

At the first step we estimate $W_2^1(\Om^\e)$-norm of $v^\e$.
Function $v^\e$ solves the boundary value problem
\begin{equation}\label{7.7}
\begin{aligned}
&\left(-\sum\limits_{i,j=1}^{2} \frac{\p\hphantom{x}}{\p x_i} A_{ij} \frac{\p\hphantom{x}}{\p x_j} + \sum\limits_{j=1}^{2} A_j\frac{\p\hphantom{x}}{\p x_j}-\frac{\p\hphantom{x}}{\p x_j}\overline{A_j} + A_0-\iu\right)v^\e=0\quad \text{in}\quad \Om^\e\setminus\widetilde{\g},
\\
&\hphantom{00}v^\e=0\quad\text{on}\quad \p\Om,\qquad v^\e=-\tu\quad\text{on}\quad \p\tht^\e,\qquad \quad [v^\e]_{\widetilde{\g}}=0,
\\
&\left(\frac{\p\hphantom{N^\e}}{\p N^\e}+a\right)v^\e=-\left(\frac{\p\hphantom{N^\e}}{\p N^\e}+a\right)\tu\quad\text{on}\quad \p\tht^\e_1,\qquad
\bigg[\frac{\p v^\e}{\p \widetilde{N}^0}\bigg]_{\widetilde{\g}}-\beta \tu\big|_{\widetilde{\g}}=0.
\end{aligned}
\end{equation}
As we see, function $v^\e$ does not satisfy homogeneous Dirichlet condition on $\p\tht^\e_0$. In order to simplify certain technical estimates, we add a special boundary corrector to $v^\e$ so that the sum vanishes on $\p\tht^\e_0$. Then employing the above boundary value problem, we shall obtain an integral identity for this sum and estimate its norm. We define the boundary corrector as follows.

We let
\begin{equation*}
\mathrm{A}(x)=
\begin{pmatrix}
A_{11}(x) & A_{12}(x)
\\
A_{12}(x) & A_{22}(x)
\end{pmatrix},\quad \mathrm{A}_k^\e:=\mathrm{A}(y_k^\e),
\end{equation*}
and for each $k\in\mathds{M}_0^\e$ by $\mathrm{Q}_k^\e$ we denote the matrix satisfying
\begin{equation}\label{7.5}
(\mathrm{Q}_k^\e)^t (\mathrm{Q}_k^\e)=(\mathrm{A}_k^\e)^{-1},\quad
(\mathrm{Q}_k^\e)^t\mathrm{A}_k^\e (\mathrm{Q}_k^\e)=\mathrm{E},
\end{equation}
where $\mathrm{E}$ is the unit matrix. Due to condition (\ref{2.3}), matrix $\mathrm{A}$ is symmetric, lower-semibounded, and bounded uniformly in $x\in\overline{\Om}$, and this is why matrix $\mathrm{Q}_k^\e$ is well-defined, symmetric, lower-semibounded, and bounded uniformly in $k$ and $\e$. Hence, we have the estimate
\begin{equation}\label{7.6}
0<C|z| \leqslant  |\mathrm{Q}_k^\e z | \leqslant C^{-1} |z|
\end{equation}
that is uniform in  $k\in\mathds{M}_0^\e$, $z\in\mathds{R}^2$, and sufficiently small $\e$.

For each $k\in\mathds{M}_0^\e$ we define the ellipses $E_r^k:=\{x: |\mathrm{Q}_k^\e\z^k|<\e r R_5\}$. Here $R_5$ is an absolute positive constant $R_5$ that exists due to (\ref{7.6}), (\ref{7.11}) and for all sufficiently small $\e$ and $k\in\mathds{M}_0^\e$
\begin{equation}\label{7.8}
\om_k^\e\subseteq E_{\eta}^k \subset E_{1}^k\subseteq B_{\frac{R_2}{2}\e}(y_k^\e).
\end{equation}

We define the function
\begin{equation}\label{7.31}
W^\e(x):=\left\{
\begin{aligned}
&\frac{1}{\ln\eta(\e)} \ln \frac{|\mathrm{Q}_k^\e\z^k|}{R_5\e}, \qquad x\in E_{1}^k\setminus E_{\eta}^k,\quad k\in\mathds{M}_0^\e,
\\
&\hphantom{-\frac{1}{\ln\eta(\e)}}1, \hphantom{ \frac{|\mathrm{Q}_k^\e\xi^k|}{R_5\e\eta}}\qquad x\in E_{\eta}^k,\hphantom{\setminus E_{1}^k,}\quad k\in\mathds{M}_0^\e,
\\
&\hphantom{-\frac{1}{\ln\eta(\e)}}0, \hphantom{ \frac{|\mathrm{Q}_k^\e\xi^k|}{R_5\e\eta }}\qquad \text{otherwise},
\end{aligned}
\right.
\end{equation}
where, we remind, $\z^k:=x-y_k^\e$. It is clear that $W^\e$ is infinitely differentiable  in $\overline{\Om}$ except the boundaries $\p E_{1}^k$ and $\p E_{\eta}^k$, $k\in\mathds{M}_0^\e$, and is continuous  in $\overline{\Om}$. Function $W^\e$ is bounded uniformly in $\overline{\Om}$ and satisfies the estimate
\begin{equation}\label{6.10}
0\leqslant W^\e\leqslant 1.
\end{equation}
An important property of $W^\e$ is that the function $\tv:=v^\e+\tu W^\e$ vanishes on $\p\tht^\e_0$.

We multiply the equation in (\ref{7.7}) by $\tv$ and integrate once by parts taking into consideration the boundary conditions for $v^\e$. Then we replace $v^\e$ by $\tv-\tu W^\e$. It yields the integral identity for $\tv$:
\begin{equation}\label{6.9}
\begin{aligned}
\fae(\tv,\tv)&+(a \tv, \tv)_{L_2(\p\tht^\e_1)} + (\b \tv, \tv)_{L_2(\widetilde{\g})}-\iu \|\tv\|_{L_2(\Om^\e)}^2
=(\b \tu, \tv)_{L_2(\widetilde{\g})}
\\
&-\left(
\left(\frac{\p\hphantom{N^\e}}{\p N^\e}+a\right)\tu,\tv
\right)_{L_2(\p\tht^\e_1)}+\fae(\tu W^\e,\tv)-\iu(\tu,\tv W^\e)_{L_2(\Om^\e)}.
\end{aligned}
\end{equation}
Let us estimate the right hand side of this identity.

By (\ref{2.12}),  $\mu(\e)\to+0$ as $\e\to+0$ and
\begin{equation}\label{7.11}
\eta(\e)=\E^{-\frac{1}{\e(\rho+\mu(\e))}}.
\end{equation}
Then it follows from (\ref{6.8}), (\ref{3.12}), (\ref{3.11}), (\ref{6.10}), and (\ref{6.3}) that
\begin{equation}\label{6.11}
\begin{aligned}
&\left|\left(
\left(\frac{\p\hphantom{N^\e}}{\p N^\e}+a\right)\tu, \tv
\right)_{L_2(\p\tht^\e_1)}\right|+\left|(\tu W^\e,\tv)_{L_2(\Om^\e)}\right|
\\
&\leqslant C\Big(\eta^{\frac{1}{2}}(|\ln\eta|^{\frac{1}{2}}+1)+\e\Big)\|f\|_{L_2(\Om)} \|v^\e\|_{W_2^1(\Om_\e)}\leqslant C\e \|f\|_{L_2(\Om)} \|v^\e\|_{W_2^1(\Om_\e)}.
\end{aligned}
\end{equation}
The definition of form $\fao^\e$ yields that
\begin{align}
&\fae(\tu W^\e,\tv)=S_1^\e+S_2^\e+S_3^\e,\label{6.13}
\\
&S_1^\e:=\sum\limits_{k\in\mathds{M}_0^\e}  ( \mathrm{A}_k^\e\tu \nabla W^\e,\nabla \tv)_{L_2(E^k_{1}\setminus E^k_{\eta})},
\quad
S_2^\e:=\sum\limits_{k\in\mathds{M}_0^\e} \sum\limits_{j=1}^{2} \left(A_j \tu \frac{\p W^\e}{\p x_j},\tv\right)_{L_2(E^k_{1}\setminus E^k_{\eta})},\nonumber
\\
&S_3^\e:=\sum\limits_{k\in\mathds{M}_0^\e}  \left( (\mathrm{A}-\mathrm{A}_k^\e) \tu \nabla  W^\e, \nabla \tv \right)_{L_2(E^k_{1}\setminus E^k_{\eta})}  + \sum\limits_{k\in\mathds{M}_0^\e}  \left(\mathrm{A} W^\e \nabla \tu, \nabla \tv \right)_{L_2(E^k_{1}\setminus E^k_{\eta})} \nonumber
\\
& \hphantom{S_2^\e:=}
  + \sum\limits_{k\in\mathds{M}_0^\e} \sum\limits_{j=1}^{2} \left(A_j W^\e\frac{\p \tu}{\p x_j}, \tv\right)_{L_2(E^k_{1}\setminus E^k_{\eta})}  + \sum\limits_{k\in\mathds{M}_0^\e} \sum\limits_{j=1}^{2} \left(W^\e \tu , A_j\frac{\p \tv}{\p x_j}  \right)_{L_2(E^k_{1}\setminus E^k_{\eta})}\nonumber \\
& \hphantom{S_2^\e:=}+ (A_0 W^\e \tu, \tv)_{L_2(E^k_{1}\setminus E^k_{\eta})}.\nonumber
\end{align}
As it follows from the definition of $W^\e$, in $E^k_{1}\setminus E^k_{\eta}$ this function satisfies the inequalities
\begin{equation}\label{6.15}
|\nabla W^\e|\leqslant \frac{C}{|\z^k||\ln\eta|},
\quad |(\mathrm{A}-\mathrm{A}_k^\e)\nabla W^\e|\leqslant \frac{C}{|\ln\eta|}.
\end{equation}
We employ these inequalities and (\ref{3.11}) with $\eta=1$, (\ref{6.3}), (\ref{6.10}) to estimate $S_3^\e$:
\begin{equation}\label{6.14}
|S_3^\e|\leqslant C\e^{\frac{1}{2}} \|f\|_{L_2(\Om)} \|v^\e\|_{W_2^1(\Om^\e)}.
\end{equation}

Lemma~\ref{lm7.2} and inequalities (\ref{3.11}) with $\eta=1$, (\ref{6.3}), and (\ref{6.15}) allow us to estimate $S_2^\e$:
\begin{equation}\label{6.16}
\begin{aligned}
|S_2^\e|\leqslant &  \left(\sum\limits_{k\in\mathds{M}_0^\e} \|v^\e\|_{L_2(E^k_{1}\setminus E^k_{\eta})}^2\right)^{\frac{1}{2}} \left(\sum\limits_{k\in\mathds{M}_0^\e} \|\tu\|_{C(\overline{B_{R_2\e}(y_k^\e)})}^2 \|\nabla W^\e\|_{L_2(E^k_{1}\setminus E^k_{\eta})}^2\right)^{\frac{1}{2}}
\\
\leqslant & C\e^{\frac{1}{2}} \| v^\e\|_{W_2^1(\Om^\e)} \|f\|_{L_2(\Om)}.
\end{aligned}
\end{equation}

Employing the definition of $W^\e$, we integrate by parts as follows:
\begin{gather}
S_1^\e=S_4^\e+S_5^\e+S_6^\e,\label{6.17}
\\
\begin{aligned}
S_4^\e:=&\frac{1}{\ln\eta}
\sum\limits_{k\in\mathds{M}_0^\e}
\left(\frac{\tu}{|(\mathrm{A}_k^\e)^{-1}\z^k|},\tv\right)_{L_2(\p E^k_{1})},\quad S_5^\e:=-\frac{1}{\ln\eta}\sum\limits_{k\in\mathds{M}_0^\e}\left( \frac{\tu}{|(\mathrm{A}_k^\e)^{-1}\z^k|},\tv\right)_{L_2(\p E^k_{\eta})},
\\
S_6^\e:=&-\sum\limits_{k\in\mathds{M}_0^\e}(\mathrm{A}_k^\e\nabla \tu, \tv\nabla W^\e)_{L_2(E^k_{1}\setminus E^k_{\eta})}=-\sum\limits_{k\in\mathds{M}_0^\e}\left(\nabla \tu,\frac{\z^k\tv}{|(\mathrm{A}_k^\e)^{-1}\z^k|}\right)_{L_2(\p E^k_{\eta})}
\\
&+\sum\limits_{k\in\mathds{M}_0^\e}\Big((\mathrm{A}_k^\e\nabla \tu, W^\e \nabla \tv)_{L_2(E^k_{1}\setminus E^k_{\eta})}+(\mathrm{div}\, \mathrm{A}_k^\e\nabla \tu, \tv W^\e)_{L_2(E^k_{1}\setminus E^k_{\eta})}\Big)
\end{aligned}\nonumber
\end{gather}
We make one more integration by part similar to (\ref{3.10b}):
\begin{align*}
0&=\sum\limits_{k\in\mathds{M}_0^\e} \frac{1}{\ln\eta} \int\limits_{B_{b_*R_2\e}(y_k^\e)\setminus E^k_{1}} \tu \overline{\tv} \mathrm{div}\, \mathrm{A}_k^\e \nabla \ln|\mathrm{Q}_k^\e\z^k|\di x=-S_4^\e+S_7^\e-S_8^\e,
\\
S_7^\e:&=\frac{1}{\ln\eta}\sum\limits_{k\in\mathds{M}_0^\e} (\tu, \phi_k^\e \tv)_{L_2(\p B_{b_*R_2\e}(y_k^\e))},
\quad
\phi_k^\e:=\nu\cdot \mathrm{A}_k^\e \nabla \ln|\mathrm{Q}_k^\e\z^k|,
\\
S_8^\e:&=\frac{1}{\ln\eta}\sum\limits_{k\in\mathds{M}_0^\e} (\nabla \tu, \tv \mathrm{A}_k^\e \nabla \ln|\mathrm{Q}_k^\e\z^k|)_{L_2(B_{b_*R_2\e}(y_k^\e))\setminus E^k_{1})}
\\
&\hphantom{=}
+\frac{1}{\ln\eta}\sum\limits_{k\in\mathds{M}_0^\e}
(\tu  \mathrm{A}_k^\e \nabla \ln|\mathrm{Q}_k^\e\z^k|, \nabla \tv)_{L_2(B_{b_*R_2\e}(y_k^\e)\setminus E^k_{1})},
\end{align*}
where $\nu$ is the outward normal to $\p B_{b_*R_2\e}(y^\e_k)$. Thus,
\begin{equation}\label{6.25}
S_1^\e=S_7^\e-S_8^\e+S_5^\e+S_6^\e.
\end{equation}

It follows from
(\ref{3.11}) with $\eta=1$, (\ref{6.10}), (\ref{6.3}) and the estimate
\begin{equation*}
|\nabla  \ln|\mathrm{Q}_k^\e\z^k|| \leqslant C\e^{-1}\quad \text{in}\quad B_{b_*R_2\e}(y_k^\e)\setminus E^k_{1}
\end{equation*}
that
\begin{equation}\label{6.23}
|S_6^\e|+|S_8^\e|\leqslant C\e^{\frac{1}{2}} \|f\|_{L_2(\Om)}\|\tv\|_{W_2^1(\Om^\e)}.
\end{equation}

To estimate $S_5^\e$, we first observe that since $\tv$ vanishes on $\p\om_k^\e$, $k\in\mathds{M}_0^\e$, we can extend it by zero inside $\om_k^\e$ keeping its $L_2$- and $W_2^1$-norms. By Assumption \ref{B1} we have that for some $R_6<1$ the inclusion $E_{R_6\eta}\subset\om_k^\e$  holds true and thus $\tv=0$ on $\p E_{R_6\eta}\subset\om_k^\e$. We then pass to the variables $\z^k\mapsto\mathrm{Q}_k^\e \z^k\e^{-1}\eta^{-1}$ and employ the estimate
\begin{equation*}
  \|v\|_{L_2(\p B_{R_5}(0))}\leqslant C\|\nabla v\|_{L_2(B_{R_5}(0)\setminus B_{R_6 R_5}(0))}
\end{equation*}
valid for each $v\in \Ho^1\big(B_{R_5}(0)\setminus B_{R_5 R_6}(0),\p B_{R_5}(0)\big)$. Finally it yields
\begin{equation}\label{6.18}
\|\tv\|_{L_2(\p E_{\eta}^k)}\leqslant C\e^{\frac{1}{2}}\eta^{\frac{1}{2}} \|\nabla \tv\|_{L_2(E_{\eta}^k\setminus \om^\e_k)},\quad k\in\mathds{M}_0^\e.
\end{equation}
Since $|(\mathrm{Q}_k^\e)^2\z^k|\geqslant C\e\eta$ on $\p E^k_{\eta}$, $C>0$, by (\ref{6.18}), (\ref{6.3}), and Lemma~\ref{lm7.2} we can estimate $S_5^\e$:
\begin{equation}\label{6.21}
|S_5^\e|\leqslant  C \e^{\frac{1}{2}}(\rho+\mu) \|f\|_{L_2(\Om)} \|\tv\|_{W_2^1(\Om^\e)}.
\end{equation}
Due to (\ref{6.11}), (\ref{6.13}), (\ref{6.14}), (\ref{6.16}), (\ref{6.17}), (\ref{6.23}), (\ref{6.21}) it remains to estimate the sum $(\b \tu, \tv)_{L_2(\widetilde{\g})}+S_7^\e$ in order to have the final inequality for the right hand side in (\ref{6.9}).

By straightforward calculations one can make sure that
\begin{equation*}
\int\limits_{\p B_{b_*R_2\e}(y^\e_k)} \phi^\e_k\di s= \frac{2\pi}{\det \mathrm{A}_k^\e}.
\end{equation*}
Proceeding then as in the proof of Lemma~\ref{lm3.7}, in $S_7^\e$ we can replace $\phi^\e_k$ by its mean value over $\p B^k_{b_*}$:
\begin{equation*}
\Big|S_7^\e-\frac{1}{\e\ln\eta}\sum\limits_{k\in\mathds{M}_0^\e} \frac{2}{(b+1)R_2\det \mathrm{A}_k^\e}(\tu,\tv)_{L_2(\p B_{b_*R_2\e}(y_k^\e))}
\Big|\leqslant C\e^{\frac{1}{2}}\|f\|_{L_2(\Om)}\|\tv\|_{W_2^1(\Om^\e)}.
\end{equation*}
Arguing as in the proof of (\ref{6.1}) and applying then Lemma~\ref{lm6.3}, one can make sure that
\begin{align*}
\Big|\frac{1}{\e\ln\eta}\sum\limits_{k\in\mathds{M}_0^\e} \frac{2}{(b+1)R_2\det \mathrm{A}_k^\e}(\tu,\tv)_{L_2(\p B_{b_*R_2\e}(y_k^\e))}&-(\b\tu,\tv)_{L_2(\widetilde{\g})}
\Big|
\\
&\leqslant C(\rho+\mu)\vk \|f\|_{L_2(\Om)}\|\tv\|_{W_2^1(\Om^\e)}.
\end{align*}
Two latter estimates, (\ref{6.11}), (\ref{6.13}), (\ref{6.14}), (\ref{6.16}), (\ref{6.17}), (\ref{6.25}), (\ref{6.23}), (\ref{6.21}), and Lemma~\ref{lm1.3} yield
\begin{equation}\label{6.24}
\|\tv\|_{W_2^1(\Om^\e)}\leqslant C\big(\e^{\frac{1}{2}}+\vk(\rho+\mu)\big)\|f\|_{L_2(\Om)}.
\end{equation}

It follows from Lemma~\ref{lm7.2} and (\ref{6.4}) that
\begin{equation*}
\left(\sum\limits_{k\in\mathds{M}_0^\e} \|u^0-\tu\|_{C(\overline{B_{R_2\e}(y_k^\e}))}^2\right)^{\frac{1}{2}}\leqslant C\|f\|_{L_2(\Om)}.
\end{equation*}
Hence, due to  the definition of $W^\e$ and (\ref{6.3}),
\begin{equation}
\begin{aligned}
\Bigg(\sum\limits_{k\in\mathds{M}_0^\e} & \|(u^0-\tu)\nabla W^\e\|_{L_2(E_{1}^k)}^2\Bigg)^{\frac{1}{2}}
\\
&\leqslant
\left(\sum\limits_{k\in\mathds{M}_0^\e} \|u^0-\tu\|_{C(\overline{B_{R_2\e}(y_k^\e)}}^2 \|\nabla W^\e\|_{L_2(E_{1}^k)}^2 \right)^{\frac{1}{2}}
\leqslant C\e^{\frac{1}{2}}\|f\|_{L_2(\Om)}.
\end{aligned}\label{6.26}
\end{equation}
Employing   (\ref{3.11}), (\ref{1.7}), and (\ref{6.10}), we obtain
\begin{equation*}
\left(\sum\limits_{k\in\mathds{M}_0^\e} \left(\|W^\e\nabla(u^0-\tu) \|_{L_2(E_{1}^k)}^2+ \|W^\e (u^0-\tu) \|_{L_2(E_{1}^k)}^2\right)\right)^{\frac{1}{2}}
\leqslant C\e^{\frac{1}{2}}\|f\|_{L_2(\Om)}.
\end{equation*}
Together with (\ref{6.24}), (\ref{6.26}) it yields
\begin{equation}\label{6.21a}
\|u^\e-(1-W^\e)u^0\|_{W_2^1(\Om^\e)}\leqslant C\big(\e^{\frac{1}{2}}+(\rho+\mu)\vk)\|f\|_{L_2(\Om)}
\end{equation}
that proves (\ref{2.21}).

Since $W^\e u^0$ vanishes outside $\p E^k_{1}$, $k\in\mathds{M}_0^\e$, by (\ref{6.10}), (\ref{3.11}), (\ref{1.7}) we get
\begin{equation*}
\|W^\e u^0\|_{L_2(\Om)} \leqslant \left(\sum\limits_{k\in\mathds{M}_0^\e}\|u^0\|_{L_2(E_1^k)}^2\right)^{\frac{1}{2}}
\leqslant C\e^{\frac{1}{2}} \|f\|_{L_2(\Om)}.
\end{equation*}
Together with  (\ref{6.21a}) it implies
\begin{equation}\label{7.30}
\|u^\e-u^0\|_{L_2(\Om^\e)}\leqslant C \big(\e^{\frac{1}{2}}+(\rho+\mu)\varkappa\big) \|f\|_{L_2(\Om)},
\end{equation}
and therefore (\ref{2.18a}) holds true.

To prove (\ref{2.19}), now it is sufficient to employ the obvious estimate
\begin{equation*}
\|(\opo{\beta}-\iu)^{-1}f-(\opo{\beta_0}-\iu)^{-1}f\|_{W_2^1(\Om)} \leqslant C\mu\|f\|_{L_2(\Om)}.
\end{equation*}

As $\rho=0$, by inequalities (\ref{3.11}), (\ref{1.5}) and Lemma~\ref{lm7.2} we get
\begin{align*}
\|\nabla W^\e& u^0\|_{L_2(\Om)}\leqslant \|W^\e\nabla u^0\|_{L_2(\Om)} + \|u^0\nabla W^\e\|_{L_2(\Om)}
\\
&\leqslant C \left(\e^{\frac{1}{2}} \|f\|_{L_2(\Om)} + \left(\sum\limits_{k\in\mathds{M}_0^\e}  \|u^0\nabla W^\e\|_{L_2(B_{R_2\e}(y_k^\e))}^2\right)^{\frac{1}{2}}
\right)
\\
&\leqslant C\left(\e^{\frac{1}{2}}\|f\|_{L_2(\Om)}+ \left(\sum\limits_{k\in\mathds{M}_0^\e} \|u^0\|_{C(\overline{B_{R_2\e}(y_k^\e)})}^2 \|\nabla W^\e\|_{L_2(B_{R_2\e}(y_k^\e))}^2\right)^{\frac{1}{2}}
\right)
\\
&\leqslant C\left(\e^{\frac{1}{2}}\|f\|_{L_2(\Om)}+ \e^{\frac{1}{2}}\mu^{\frac{1}{2}} \left(\sum\limits_{k\in\mathds{M}_0^\e} \|u^0\|_{C(\overline{B_{R_2\e}(y_k^\e)})}^2\right)^{\frac{1}{2}}
\right)
\leqslant C(\e^{\frac{1}{2}}+\mu^{\frac{1}{2}})\|f\|_{L_2(\Om)}.
\end{align*}
The obtained estimate and (\ref{6.21a}) with $\rho=0$ follow
\begin{equation*}
\|u^\e-u^0\|_{L_2(\Om^\e)}\leqslant C(\e^{\frac{1}{2}}+\mu^{\frac{1}{2}})\|f\|_{L_2(\Om)}
\end{equation*}
and it proves  (\ref{2.20}). The proof is complete.

\section{Spectrum}

In this section we   prove   Theorem~\ref{th2.5}. In $\tht^\e$ we introduce operator $\opt$ acting as $-\D$ subject to the Dirichlet condition; the associated form is $(\nabla u, \nabla v)_{L_2(\tht^\e)}$ on $\Ho^1(\tht_\e)$. Employing minimax principle and Assumption \ref{B1}, one can easily make sure that
\begin{equation}\label{8.4}
\inf\spec(\opt)\geqslant C\e^{-2}\eta^{-2}(\e),
\end{equation}
where $\spec(\opt)$ denotes the spectrum of operator $\opt$. Thus, we have the estimate
\begin{equation}\label{8.1}
\|(\opt-\iu)^{-1}\|_{L_2(\tht^\e)\to L_2(\tht^\e)}\leqslant C\e^{2}\eta^{2}(\e).
\end{equation}
Assuming the hypothesis of one of Theorems~\ref{th2.1},~\ref{th2.4},~\ref{th2.2},~\ref{th2.3}, by $\opo{*}$ we denote the corresponding homogenized operator. Estimates (\ref{3.11}) and (\ref{3.21}) imply
\begin{equation*}
\|(\opo{*}-\iu)^{-1} \|_{L_2(\Om)\to L_2(\tht^\e)}\leqslant C\e,
\end{equation*}
where $C$ is a constant independent of $\e$. Since $L_2(\Om)=L_2(\Om^\e)\oplus L_2(\tht^\e)$, the latter estimate and (\ref{8.1}) yield
\begin{equation*}
\|(\ope\oplus\opt-\iu)^{-1}-(\opo{*}-\iu)^{-1}\|_{L_2(\Om)\to L_2(\Om)}\to 0,\quad \e\to+0.
\end{equation*}
By \cite[Ch. V\!I\!I\!I, Sec. 7, Ths. V\!I\!I\!I.23, V\!I\!I\!I.24]{RS1} it follows the convergence of the spectrum of $\ope\oplus\opt$ to that of $\opo{*}$. And now it remains to employ (\ref{8.4}) to complete the proof of Theorem~\ref{th2.5}.

\section{Sharpness of estimates}

In this section we discuss the sharpness of the estimates established in Theorems~\ref{th2.1}--\ref{th2.3}. We show that estimates (\ref{2.5a}), (\ref{2.21}), (\ref{2.13c}) are order sharp, while other estimates are close to being sharp.

In order to study the sharpness, we need to know how the difference of the perturbed and homogenized resolvents behaves for at least one model fitting our assumptions. Of course, there is no chance to find the perturbed resolvents explicitly. Instead of this, we choose the perturbed operator so that it is possible to construct the asymptotic expansions for its resolvent. Namely, suppose that for some $f\in L_2(\Om)$ we know the first term of the asymptotic expansion for $u^\e:=(\ope-\iu)^{-1}f$:
\begin{equation}\label{8.1a}
u^\e(x)=u^0(x)+u^1(x,\e)+\ldots,
\end{equation}
where $u^0$ is the action of the corresponding homogenized resolvent on $f$, and $u^1$ is a some function. Assume also that the above expansion is true in $W_2^1(\Om^\e)$-norm. Then
\begin{equation*}
u^\e-u^0=u^1(x,\e)+\ldots,
\end{equation*}
and in the left hand side we have in fact the difference of perturbed and homogenized resolvents. And if the norm $\|u^1\|_{W_2^1(\Om^\e)}$ has the same smallness order as the corresponding estimate in Theorems~\ref{th2.1}--\ref{th2.3} states, then this estimate is order sharp.

In what follows we construct asymptotics (\ref{8.1a}) under the hypothesis of each of Theorems~\ref{th2.1}--\ref{th2.3}. To construct the required asymptotics we choose a very simple model. Namely, we assume that the width of strip $\Om$ is $d=\pi$ and curve $\g$ is just the straight line $\{x:\, x_2=\frac{\pi}{2}\}$. The perforation is periodic: $s^\e_k:=3\e k$, $k\in\mathds{Z}$, and $y_k^\e:=\{x:\, x_1=3\e\ k, \, x_2=\frac{\pi}{2}\}$. All the holes are just the unit balls: $\om_k:=B_1(0)$, and thus $\om_k^\e:=B_{\e\eta(\e)}(y_k^\e)$. Then Assumption~\ref{B1} is obviously satisfied with
$R_2=\frac{5}{4}$, $b=\frac{6}{5}$, $L=2\pi$. Assumption~\ref{B2} is also true and we can find functions $X_k$ explicitly:
\begin{equation*}
X_k(z)=\nabla_z\ln|z|=\frac{1}{|z|^2}z,\quad z=(z_1,z_2), \quad \vp_k\equiv \frac{11}{8}.
\end{equation*}

As the operator, we choose the Laplacian, i.e., the differential expression in (\ref{2.4}) is $-\D$. For simplicity, we impose the same condition on the boundaries of all the holes. It is either the Dirichlet condition ($\mathds{M}_0^\e=\mathds{M}^\e=\mathds{Z}$, $\mathds{M}_1^\e=\emptyset$), or the Robin condition
$\left(\frac{\p\hphantom{\nu}}{\p\nu}+a\right)u=0$ with constant $a$ ($\mathds{M}_1^\e=\mathds{M}^\e=\mathds{Z}$, $\mathds{M}_0^\e=\emptyset$).

The constructing of asymptotics consists of formal constructing and estimating error terms. The latter is very simple for our problem thanks to Lemma~\ref{lm1.3}. Once we construct a formal asymptotic solution to the equation $(\ope-\iu)u^\e=f$, it provides the asymptotics for the exact solution. This is why in what follows we dwell only on the formal constructing of first terms in the asymptotics. Our approach is based on
the method of matching asymptotic expansions \cite{Il}
and the multiscale method \cite{MS}. Further terms needed for estimating the error terms can be constructed in the same way.

In order to simplify our considerations, we shall also assume that
\begin{equation}\label{8.2}
u^0(x):=\chi_2(x_1) U(x_2),
\end{equation}
where $\chi_2$ is an infinitely differentiable cut-off function equalling one as $|x_1|<1$ and vanishing as $|x_2|>2$. Function $U$ vanishes as $x_2=0$ and $x_2=\pi$. It is infinitely differentiable in each of the segments $[0,\frac{\pi}{2}]$ and $[\frac{\pi}{2},\pi]$ and is continuous in $[0,\pi]$. In the vicinity of the point $x_2=\frac{\pi}{2}$ function $U$ is point-wise linear:
\begin{equation}\label{8.3}
\begin{aligned}
&U(x_2)=h_1^+\left(x_2-\frac{\pi}{2}\right)+h_0,\quad \frac{\pi}{2}\leqslant x_2\leqslant \frac{3\pi}{4},
\quad U(x_2)=h_1^-\left(x_2-\frac{\pi}{2}\right)+h_0,\quad \frac{\pi}{4}\leqslant x_2\leqslant \frac{\pi}{2},
\end{aligned}
\end{equation}
where $h_0$, $h_1^\pm$ are some constants which will be specified below.

We define $f$ as follows:
\begin{equation}\label{8.4a}
f(x)=(-\D-\iu) u^0(x),
\end{equation}
and this identity is understood pointwise in $\Om\setminus\g$. It is clear that function $f$ belongs to $L_2(\Om)$.

It should also stressed that although we have supposed that
curve $\g$ is infinite, all our arguments in what follows can be adapted easily for a finite curve $\g$.

\subsection{Dirichlet condition}

In this subsection we study the sharpness of the estimates in Theorems~\ref{th2.1},~\ref{th2.3}. On the boundaries of the holes we impose the Dirichlet condition.

We begin with constructing asymptotics under the hypothesis of Theorem~\ref{th2.1}. It is clear that our simple model described above satisfies Assumption~\ref{A5} with $R_3=5/2$. We also suppose that identity (\ref{2.1ab}) holds true. The homogenized operator has the Dirichlet condition on $\g$ and $u^0$ should vanish on $\g$. This is why we let $h_0=0$ in (\ref{8.3}). For simplicity we also let $h_1^\pm=\pm 1$. Our aim is to construct formally the asymptotic expansion for $u^\e:=(\ope-\iu)^{-1}f$, where $f$ is defined by (\ref{8.4a}). In what follows it is more convenient to regard $u^\e$ as the generalized solution to the boundary value problem
\begin{equation}\label{8.5}
(-\D-\iu)u^\e=f\quad\text{in}\quad \Om^\e,\qquad u^\e=0\quad\text{on}\quad\p\Om^\e.
\end{equation}
We construct the asymptotics for $u^\e$ by the method of matching asymptotic expansions as a combination of external and internal layers. The external one reads as
\begin{equation}\label{8.13}
u^\e_{\ex}(x)=u^0(x)+\e u^1(x)+\ldots
\end{equation}
Hereinafter by ``\ldots'' we denote next terms in various asymptotics.

Function $u^0$ does not vanish on $\p\tht^\e$. To obtain the required boundary condition for $u^\e$ on $\p\tht^\e$, we introduce rescaled variables by the rule $\xi=(\xi_1,\xi_2)$, $\xi_1:=x_1\e^{-1}$, $\xi_2:=(x_2-\frac{\pi}{2})\e^{-1}$. In the vicinity of $\g$ we construct an internal expansion:
\begin{equation}\label{8.6}
u^\e_{\bl}(x)=\e v^1(\xi,x_1)+\ldots
\end{equation}
We substitute this ans\"atz into boundary value problem (\ref{8.5}) and equate the coefficients at the like powers of $\e$. It leads to the boundary value problem for $v^1$:
\begin{equation}\label{8.7}
\D_{\xi}v^1=0\quad \text{in}\quad\mathds{R}^2\setminus\p\tht_\eta,\quad v^1=0\quad \text{on}\quad\p\tht_\eta.
\end{equation}
Here
\begin{equation*}
\tht_\eta:=\bigcup\limits_{k\in\mathds{Z}} \om_{\eta,k},\quad \om_{\eta,k}:=\{\xi:\, |\xi-(3\pi k,0)|<\eta\}.
\end{equation*}
Problem (\ref{8.7}) is written in terms of variables $\xi$; variable $x_1$ is regarded as an additional parameter not even explicitly involved in the formulation of the problem.

In accordance with the method of matching asymptotic expansions, the leading term in the asymptotics for $v_1$ as $\xi_2\to+\infty$ should coincide with similar term in the asymptotics for $u^0$ as $x_2\to0$. Since in our case
\begin{equation}\label{8.8a}
u^0(x)=\left|x_2-\frac{\pi}{2}\right|\chi_2(x_1),\quad \left|x_2-\frac{\pi}{2}\right|<\frac{\pi}{4},
\end{equation}
we conclude that $v_1$ should behave at infinity as follows:
\begin{equation}\label{8.8}
v^1(\xi,x_1)=\chi_2(x_1)|\xi_2|+\ldots,\quad \xi_2\to\pm\infty.
\end{equation}

Problem (\ref{8.7}), (\ref{8.8}) is a periodic one and this is why it is more convenient to study the corresponding problem in the periodicity cell
$\Pi_\eta:=\left\{\xi\in\mathds{R}^2:\,  |\xi_2|<\frac{3}{2}\right\}\setminus \tht_\eta$.
This problem reads as
\begin{equation}\label{8.9}
\D_\xi v^1=0\quad\text{in}\quad \Pi_\eta,\qquad v^1=0\quad \text{on}\quad \p\tht^0_\eta,
\end{equation}
while on the lateral boundaries of $\Pi_\eta$ periodic boundary conditions are imposed.  At infinity, we still assume asymptotics (\ref{8.8}).

To solve problem (\ref{8.9}), we first introduce an auxiliary function
\begin{equation*}
Z_0(\xi):=\frac{3}{\pi}\RE\ln 2\sin\frac{\pi}{3}(\xi_1+\iu\xi_2).
\end{equation*}
It is straightforward to check that this function is infinitely differentiable and harmonic in $\mathds{R}^2$ except the points $(3\pi k,0)$, $k\in\mathds{Z}$. At these points it has the logarithmic singularity
\begin{equation}\label{8.11a}
Z_0(\xi)=\frac{3}{\pi}\ln|\xi-(3\pi k,0)|+\frac{3}{\pi}\ln \frac{2\pi}{3} + O(|\xi-(3\pi k,0)|^2),\quad \xi\to(3\pi k,0).
\end{equation}
Function $Z_0$ is $3$-periodic in $\xi_1$. At infinity, it behaves as \begin{equation*}
Z_0(\xi)=\pm\xi_2+O(e^{-\frac{2\pi}{3}|\xi_2|}),\quad \xi_2\to\pm\infty.
\end{equation*}

\begin{lemma}\label{lm8.1a}
Problem (\ref{8.9}) has the unique periodic solution behaving at infinity as
\begin{equation}\label{8.11}
Z_\eta^D(\xi)=\pm\xi_2-\frac{3}{\pi}\ln \frac{2\pi\eta}{3}+O(e^{-\frac{2\pi}{3}|\xi_2|}),\quad \xi_2\to\pm\infty.
\end{equation}
This solution is $3$-periodic w.r.t. $\xi_1$ and can be represented as
\begin{equation}\label{8.11c}
Z_\eta^D(\xi)=Z_0(\xi)-\frac{3}{\pi} \ln \frac{2\pi\eta}{3}  + \wt{Z}_\eta^D(\xi),
\end{equation}
where function $\wt{Z}_\eta^D$ is $3$-periodic w.r.t. $\xi_1$, decays exponentially as $\xi_2\to\pm\infty$, and satisfies the uniform in $\eta$ estimate
\begin{equation}\label{8.11d}
\|\wt{Z}_\eta^D\|_{W_2^1(\Pi_\eta)}\leqslant C\eta.
\end{equation}
\end{lemma}

\begin{proof}
We construct function $Z_\eta^D$  by formula (\ref{8.11c}). For $\wt{Z}$ we obtain the boundary value problem
\begin{equation*}
\D_\xi\wt{Z}_\eta^D=0\quad\text{in}\quad\Pi_\eta,\qquad \wt{Z}_\eta^D=\eta g_\eta(\phi),
\end{equation*}
with periodic  conditions on the lateral boundaries. Here $(r,\phi)$ are polar coordinates associated with $\xi$, and function $g$ can be expressed as the sum of the series
\begin{equation*}
g_\eta(\phi)= \sum\limits_{m=2}^{\infty}\eta^{m-1} \big(a_m^c\cos m\phi+a_m^s\sin m\phi\big),
\end{equation*}
where $a_m^c$, and $a_m^s$ are some coefficients such that
\begin{equation*}
\sum\limits_{m=2}^{\infty} m^4\big(|a_m^c|^2+|a_m^s|^2\big)\leqslant C.
\end{equation*}
Then function
\begin{equation*}
G_\eta(\xi):=\eta\sum\limits_{m=2}^{\infty} \frac{\eta^{2m-1}}{r^m} \big(a_m^c\cos m\phi+a_m^s\sin m\phi\big)
\end{equation*}
is well-defined and belongs to $W_2^2(\{\xi:\,\eta<|\xi|<1\})$. Its norm in $W_2^2(\{\xi:\,\eta<|\xi|<1\})$ is bounded uniformly in $\eta$ and its trace on $\p\tht_\eta^0$ is exactly $g$.

We construct $\wt{Z}_\eta^D$ as
\begin{equation*}
\wt{Z}_\eta^D(\xi)=\eta\wh{Z}_\eta^D(\xi)+\eta\chi_1\left(\frac{r}{2}\right) G_\eta(\xi).
\end{equation*}
We recall that $\chi_1$ is the cut-off function introduced in the proof of Lemma~\ref{lm3.3b}. For function $\wh{Z}_\eta^D$ we get the boundary value problem
\begin{equation*}
\D_\xi \wh{Z}_\eta^D=-\D_\xi \chi_1\left(\frac{r}{2}\right)  G_\eta(\xi)\quad \text{in}\quad \Pi_\eta,\qquad \wh{Z}_\eta^D=0\quad\text{on}\quad \p\tht_\eta^0,
\end{equation*}
with periodic conditions on the lateral boundaries. It is straightforward to check that
\begin{gather*}
\supp\D_\xi\chi_1\left(\frac{r}{2}\right)G_\eta(\xi)\subseteq\{\xi:\,1/2\leqslant |\xi|\leqslant 1\},
\\
\int\limits_{\Pi_\eta} \D_\xi\chi_1\left(\frac{r}{2}\right)G_\eta(\xi)\di\xi =\int\limits_{\Pi_\eta} \xi_2 \D_\xi\chi_1\left(\frac{r}{2}\right)G_\eta(\xi)\di\xi=0.
\end{gather*}
Using these identities, by the technique employed in
\cite[Sect. 3]{MSb06}, \cite[Sect. 3]{VMU02}, it is possible to show that the above problem for $\wh{Z}_\eta^D$ is uniquely solvable, its solution belongs to $W_2^2(\Pi_\eta)$ and it is bounded uniformly in $\eta$ in the norm of this Sobolev space. Returning back to function $Z_\eta^D$, we complete the proof.
\end{proof}

By the proven lemma, problem (\ref{8.9}), (\ref{8.8}) is uniquely solvable and
\begin{equation}\label{8.12}
v^1(\xi,x_1)=\chi_2(x_1) Z_\eta^D(\xi).
\end{equation}
Due to (\ref{8.11}), the asymptotic behavior of $\e v^1$ is
\begin{equation*}
\e v^1(\xi,x_1)=\e\chi_2(x_1)|\xi_2|-\frac{3\e}{\pi}\ln\frac{2\pi\eta}{3} + O(\e e^{-\frac{2\pi}{3}|\xi_2|}),\quad \xi_2\to\pm\infty.
\end{equation*}
We rewrite this asymptotics in variables $x$ and by the method of matching asymptotic expansions we conclude that
\begin{equation*}
u^1(x)=-\frac{3}{\pi}\chi_2(x_1)\ln\frac{2\pi\eta}{3}+\ldots,\quad x_2\to\frac{\pi}{2}.
\end{equation*}
We also substitute (\ref{8.13}) into (\ref{8.5}) and equate the coefficients at $\e$. Together with the above asymptotics for $v^1$ it yields the boundary value problem for $u^1$:
\begin{gather}
(-\D-\iu)u^1=0\quad\text{in}\quad \Om\setminus\g,\qquad u^1=0\quad\text{on}\quad\p\Om,\label{8.14}
\\
u^1=-\frac{3}{\pi}\chi_2(x_1)\ln\frac{2\pi\eta}{3}\quad \text{on}\quad\g.\nonumber
\end{gather}
This problem is uniquely solvable.

We ``glue'' the external and internal expansions and obtain the leading terms of the asymptotics for $u^\e$:
\begin{align}
&u^\e(x)=\big(u^0(x)+\e u^1(x)\big)\big(1-\chi_2^\e(x_2)\big) +\e v^1\left(\frac{x_1}{\e},\frac{x_2-\frac{\pi}{2}}{\e},x_1\right) \chi_2^\e(x_2) + \ldots \label{8.15}
\\
&\chi^\e_2(x_2):= \chi_2\left(\left|x_2-\frac{\pi}{2}\right|\e^{-\frac{1}{2}}\right) \nonumber
\end{align}

Let us estimate from below $W_2^1(\Om^\e)$-norm of $u^\e-u^0$.
We let
\begin{equation*}
\Om^\e_1:=\Om^\e\cap\left\{x:\,|x_1|< 3\e\left[\frac{1}{3\e}-\frac{1}{2}\right]+\frac{3\e}{2},
\,\left|x_1-\frac{\pi}{2}\right|<\e^{\frac{1}{2}}\right\}
\end{equation*}
In view of  asymptotics (\ref{8.15}), the definition of $\chi_2$ and (\ref{8.8a}), (\ref{8.12}), for $x\in\Om^\e_1$ we have:
\begin{equation*}
u^\e(x)-u^0(x)= \e Z_\eta^D\left(\frac{x_1}{\e},\frac{x_2-\frac{\pi}{2}}{\e}\right) - \left|x_2-\frac{\pi}{2}\right| + \ldots
\end{equation*}
Thus,
\begin{equation}\label{8.16}
\begin{aligned}
\|u^\e-u^0\|_{W_2^1(\Om^\e)}^2
\geqslant & C \left\|\e\frac{\p\hphantom{x}}{\p x_2} Z_\eta^D
\left(\frac{x_1}{\e},\frac{x_2-\frac{\pi}{2}}{\e}\right) - \sgn\left(x_2-\frac{\pi}{2}\right)\right\|_{L_2(\Om^\e_1)}^2
\\
= & C \left\|\frac{\p Z_\eta^D}{\p\xi_2}
\left(\frac{x_1}{\e},\frac{x_2-\frac{\pi}{2}}{\e}\right) - \sgn\left(x_2-\frac{\pi}{2}\right)\right\|_{L_2(\Om^\e_1)}^2.
\end{aligned}
\end{equation}
We rewrite the last norm as the integral and pass to  variables $\xi$ introduced above. At that, we take into consideration that $Z_\eta^D$ is $3$-periodic function w.r.t. $\xi_1$:
\begin{equation}\label{8.18}
\|u^\e-u^0\|_{W_2^1(\Om^\e)}^2\geqslant C\e^2 \left(2\left[\frac{1}{3\e}-\frac{1}{2}\right]+1\right) \int\limits_{\Pi_\eta\cap\{\xi:\, |\xi_2|<\e^{-\frac{1}{2}}\}} \left|\frac{\p Z_\eta^D}{\p \xi_2}-\sgn \xi_2\right|^2\di\xi.
\end{equation}
It follows from (\ref{8.11a}), (\ref{8.11c}), (\ref{8.11d}) that
\begin{equation*}
\int\limits_{\Pi_\eta\cap\{\xi:\, |\xi_2|<\e^{-\frac{1}{2}}\}} \left|\frac{\p Z_\eta^D}{\p \xi_2}-\sgn \xi_2\right|^2\di\xi \geqslant C(|\ln\eta|+1).
\end{equation*}
It follows from two last inequalities that
\begin{equation*}
\|u^\e-u^0\|_{W_2^1(\Om^\e)}^2\geqslant C\e (|\ln\eta|+1)
\end{equation*}
and therefore, estimate (\ref{2.5a}) is order sharp.

We proceed to Theorem~\ref{th2.4}. Assume that identity (\ref{2.12}) holds true. It is easy to check that for our simple model $\a^\e\equiv 2\pi/3$ and thus, Assumption~\ref{A14} is satisfied for $\a\equiv 2\pi/3$, $\vk\equiv 0$.

Here the formal constructing follows the same lines as above. We just need to change the boundary condition for $u^0$; now it should be (\ref{2.13}) with $\b=\frac{2\pi}{3}(\rho+\mu)$. In (\ref{8.3}) we let
\begin{equation}\label{8.23}
h_0:=1,\quad h_1^+=-h_1^-:=\frac{\pi}{3}(\rho+\mu).
\end{equation}
Then for $x_2$ close to $\frac{\pi}{2}$
\begin{equation}\label{8.22}
u^0(x)=\chi_2(x_1) \left(\pm\frac{\pi}{3}(\rho+\mu)\left(x_2-\frac{\pi}{2}\right)+1\right)
= \chi_2(x_1) \left(\pm\frac{\pi}{3}(\rho+\mu)\e\xi_2+1\right).
\end{equation}
The internal layer is again introduced by (\ref{8.6}), while the formula for $v^1$ reads as
\begin{equation}\label{8.24}
v^1(\xi,x_1)=\frac{\pi}{3}(\rho+\mu) \chi_2(x_1)Z_\eta^D(\xi).
\end{equation}
Hence, by (\ref{8.11}) and (\ref{7.11}), we see that as $\xi_2\to\pm\infty$
\begin{equation*}
\e v^1(\xi,x_1)=\chi_2(x_1) \left(\pm \frac{\pi}{3}(\rho+\mu) \xi_2 +1- \e (\rho+\mu) \ln \frac{2\pi}{3} + O\big(\e (\rho+\mu) e^{-\frac{2\pi}{3}|\xi_2|}\big)\right).
\end{equation*}
Comparing these identities and (\ref{8.22}), by the method of matching asymptotic expansions we arrive at the boundary condition for $u^1$:
\begin{equation}\label{8.21}
u^1=-(\rho+\mu) \ln \frac{2\pi}{3}\quad\text{on}\quad\g.
\end{equation}
The equation and boundary condition on $\p\Om$ are the same as in (\ref{8.14}).

The leading terms of the asymptotics for $u^\e$ are determined by (\ref{8.15}) but with $u^0$ given by (\ref{8.2}), (\ref{8.3}), (\ref{8.23}), $v^1$ given by (\ref{8.24}) and $u^1$ being the solution to (\ref{8.14}), (\ref{8.21}).

Let us check the sharpness  of estimate (\ref{2.21}). We denote
\begin{equation*}
\Om^\e_2:=\Om^\e\cap\left\{x:\,|x_1|< 3\e\left[\frac{1}{3\e}-\frac{1}{2}\right]+\frac{3\e}{2},
\,\e<\left|x_1-\frac{\pi}{2}\right|<\e^{\frac{1}{2}}\right\}.
\end{equation*}
In our case, matrix $\mathrm{Q}^\e_k$ introduced in (\ref{7.5}) is the unit one: $\mathrm{Q}^\e_k=\mathrm{E}$ and (\ref{7.8}) is satisfied with $R_5=\frac{5}{8}$. Then it follows from definition (\ref{7.31}) of $W^\e$ that this function vanishes as $x\in\Om_2^\e$. Hence, by (\ref{8.15}), for $x\in\Om_2^\e$ we get
\begin{equation*}
u^\e(x)-(1-W^\e(x))u^0(x)=\frac{\pi}{3}(\rho+\mu)\left( \e Z_\eta^D\left(\frac{x_1}{\e},\frac{x_2-\frac{\pi}{2}}{\e}\right) - \left|x_2-\frac{\pi}{2}\right| \right) + \ldots
\end{equation*}
Proceeding as in (\ref{8.16}), (\ref{8.18}), it is straightforward to check that
\begin{equation*}
\|u^\e-(1-W^\e)u^0\|_{W_2^1(\Om^\e)}\geqslant C(\rho+\mu)\e.
\end{equation*}
If $\rho\not=0$, the obtained estimate means that the term $\e^{\frac{1}{2}}$ in the right hand side (\ref{2.21}) is order sharp. The second term $(\rho+\mu)\vk$, characterizes the non-periodicity of holes distribution. This is why we can not provide any example proving the sharpness of this term. At the same time, this term comes from Lemma~\ref{lm6.3}. All the estimates in the proof of this lemma are sharp and this is why the term $(\rho+\mu)\vk$ in (\ref{2.21}) is order sharp and it proves the sharpness of estimate (\ref{2.21}).

Estimates (\ref{2.18a}), (\ref{2.19}) are not order sharp since while proving them we have employed quite rough estimate. In particular, in the proof of (\ref{7.30}) we have estimated $L_2(\Om^\e)$-norm by $W_2^1(\Om^\e)$-norm. At the same time, by (\ref{8.15}), for $|x_2-\frac{\pi}{2}|>\pi/3$,
\begin{equation*}
u^\e(x)-u^0(x)=\e u^1(x)+\ldots
\end{equation*}
and thus
\begin{equation*}
\|u^\e-u^0\|_{W_2^1(\Om^\e)}\geqslant C\e.
\end{equation*}
In estimates (\ref{2.18a}), (\ref{2.19}) we have $\e^{\frac{1}{2}}$ instead of $\e$ and in this sense this term is not far from being sharp.
Since in (\ref{2.18a}), (\ref{2.19}) we estimate $L_2$-norm, the above arguments on sharpness of $\vk$ fail and we do not know whether this term is sharp or not. The term $\mu$ in (\ref{2.19}) is obviously sharp since operator $\opo{\b}$ depends holomorphically on $\mu$.

We proceed to the case $\rho=0$; here we study the sharpness of estimate (\ref{2.20}). We first note that as $\rho=0$,
\begin{equation*}
\|(\opo{b}-\iu)^{-1}-(\opo{}-\iu)^{-1}\|_{L_2(\Om)\to W_2^1(\Om^\e)}=O(\mu)
\end{equation*}
and this is why  in (\ref{2.20}) we can replace $\opo{}$ by $\opo{\b}$ and study then the sharpness of the obtained estimate with $\opo{\b}$.

We define $u^0$ by formulae (\ref{8.2}), (\ref{8.3}), (\ref{8.23}) with $\rho=0$. It follows from (\ref{8.15}), (\ref{8.24}) that for $x\in\Om_1^\e$
\begin{equation*}
u^\e(x)- u^0(x)=\frac{\pi}{3} \mu \left( \e Z_\eta^D\left(\frac{x_1}{\e},\frac{x_2-\frac{\pi}{2}}{\e}\right) - \left|x_2-\frac{\pi}{2}\right| \right) + \ldots
\end{equation*}
Proceeding then as in (\ref{8.16}), (\ref{8.18}), we obtain
\begin{equation*}
\|u^\e-u^0\|_{W_2^1(\Om^\e)}^2\geqslant C\mu
\end{equation*}
and it proves the sharpness of the term $\mu^{\frac{1}{2}}$ in (\ref{2.20}). We can not prove that the second term $\e^{\frac{1}{2}}$ is sharp. Moreover, we conjecture that the sharp estimate should involve the term $\mu^{\frac{1}{2}}$ only, while $\e^{\frac{1}{2}}$ should be absent. Unfortunately, our technique does not allow us to prove such estimate.

\subsection{Robin condition}

In this subsection we study the sharpness of the estimates in Theorems~\ref{th2.2},~\ref{th2.3}. In our simple model we impose Robin condition
$\left(\frac{\p\hphantom{\nu}}{\p\nu}+a\right)u=0$ with constant $a$ on the boundaries of the holes. Here the constructing of asymptotics follow the same ideas as for the Dirichlet case.

We begin with Theorem~\ref{th2.3}. Assume that $a\not\equiv0$, $\eta=const$. Function $\a^\e(s)$ introduced in Theorem~\ref{th2.3} is constant $\a^\e\equiv \frac{2\pi\eta}{3}$ and Assumption~\ref{A10} is satisfied with $\a\equiv \frac{2\pi\eta}{3}$, $\vk\equiv0$. In (\ref{8.3}) we let $h_0:=1$, $h_1^+=-h_1^-:=\frac{\pi a\eta}{3}$.

We introduce the external and internal layers as
\begin{equation}\label{8.26}
u^\e_{\ex}(x)=u^0(x)+\ldots,\quad u^\e_{\bl}(x)=v^0(\xi,x_1)+\e v^1(\xi,x_1)+\ldots
\end{equation}
We substitute the internal layer into (\ref{8.5}) and equate the coefficients at the like powers of $\e$. It leads to the boundary value problem for $v^0$, $v^1$:
\begin{align}\label{8.27}
&\D_\xi v^0=0\quad\text{in}\quad \mathds{R}^2\setminus\tht_\eta,\qquad \frac{\p v^0}{\p\nu}=0\quad\text{on}\quad \p\tht_\eta,
\\
\label{8.28}
&\D_\xi v^1=0\quad\text{in}\quad \mathds{R}^2\setminus\tht_\eta,\qquad \frac{\p v^1}{\p\nu}=-a v^0\quad\text{on}\quad \p\tht_\eta,
\end{align}
where $\nu$ is the inward normal to $\p\tht_\eta$. To determine the behavior at infinity for $v^0$, $v^1$, for $x_2$ close to $\frac{\pi}{2}$, we rewrite the formula for $u^0$ in terms of variables $\xi$ and compare the obtained expression with the internal layers (\ref{8.26}). It implies:
\begin{align}\label{8.29}
& v^0=\chi_2(x_1)+\ldots,\hphantom{\pm \frac{\pi a\eta}{3}\xi_2} \quad\xi_2\to\pm\infty,
\\
&v^1=\pm \frac{\pi a\eta}{3}\chi_2(x_1)\xi_2+\ldots,  \quad\xi_2\to\pm\infty. \label{8.30}
\end{align}

Problem (\ref{8.27}), (\ref{8.29}) has a constant solution:
$v^0(\xi,x_1)=\chi_2(x_1)$.

To solve the problem for $v^1$, we need an auxiliary lemma.

\begin{lemma}\label{lm8.2}
The problem
\begin{equation*}
\D_\xi Z_\eta^R=0\quad\text{in}\quad \Pi_\eta,\qquad \frac{\p Z_\eta^R}{\p\nu}=\frac{3}{\pi\eta}\quad\text{on}\quad \p\tht_\eta,
\end{equation*}
with periodic conditions on the lateral boundaries has the unique solution behaving at infinity as
\begin{equation*}
Z_\eta^R(\xi)=\pm\xi_2 + O(e^{-\frac{2\pi}{3}|\xi_2|}),\quad \xi_2\to\pm\infty.
\end{equation*}
This solution can be represented as
$Z_\eta^N(\xi)=Z_0(\xi)+\wt{Z}_\eta^R(\xi)$,
where function $\wt{Z}_\eta^R$ is $3$-periodic w.r.t. $\xi_1$, decays exponentially as $\xi_2\to\pm\infty$, and satisfies the uniform in $\eta$ estimate
\begin{equation*}
\|\wt{Z}_\eta^R\|_{W_2^1(\Pi_\eta)}\leqslant C\eta.
\end{equation*}
\end{lemma}

The proof of this lemma is similar to that of Lemma~\ref{lm8.1a}.

The above lemma provides the solution to problem (\ref{8.28}), (\ref{8.30}):
\begin{equation*}
v^1(\xi,x_1)=-\frac{\pi a\eta}{3}\chi_2(x_1) Z_\eta^R(\xi).
\end{equation*}
The asymptotics for $u^\e$ reads as
\begin{equation}\label{8.36}
u^\e(x)=u^0(x) \left(1-\chi_2^\e(x_2)\right) +\chi_2(x_1)\left(1- \frac{\e\pi a\eta}{3} Z_\eta^R \left(\frac{x_1}{\e},\frac{x_2-\frac{\pi}{2}}{\e}\right)\right) \chi_2^\e(x_2) + \ldots
\end{equation}
For $x\in\Om_1^\e$ we get
\begin{equation*}
u^\e(x)-u^0(x)=-\frac{\e\pi a\eta}{3} Z_\eta^R \left(\frac{x_1}{\e},\frac{x_2-\frac{\pi}{2}}{\e}\right)+\ldots
\end{equation*}
As in (\ref{8.16}), (\ref{8.18}) we obtain
\begin{equation*}
\|u^\e-u^0\|_{W_2^1(\Om^\e)}^2 \geqslant C \left\|\frac{\p Z_\eta^R}{\p \xi_2}\left(\frac{x_1}{\e},\frac{x_2-\frac{\pi}{2}}{\e}\right)\right\|_{L_2(\Om^\e)}^2\geqslant C\e,
\end{equation*}
and it proves the sharpness of the term $\e^{\frac{1}{2}}$. The sharpness of the term $\vk$ is justified by the same arguments as for the estimates in Theorem~\ref{th2.4}.

We proceed to Theorem~\ref{th2.2}. Here we can not construct an example justifying the sharpness. However, estimates (\ref{2.13a}), (\ref{2.13b}) are not far from being sharp. Namely, suppose that $a\not\equiv0$, $\eta\to+0$. In (\ref{8.3}) we let $h_0:=1$, $h_1^\pm:=0$. Then by analogy with the above constructions one can get the asymptotics for $u^\e$ similar to (\ref{8.36}):
\begin{align*}
u^\e(x)=&\big(u^0(x)+\eta u^1(x)\big) \big(1-\chi_2^\e(x_2)\big)
 \\
&+ \left(\chi_2(x_1)\left(
1- \frac{\e\pi a\eta}{3} Z_\eta^R \left(\frac{x_1}{\e},\frac{x_2-\frac{\pi}{2}}{\e}\right)
\right)+\eta u^1(x_1,0)\right)\chi_2^\e(x_2) +\ldots,
\end{align*}
where $u^1$ solves boundary value problem (\ref{8.14}) with the boundary conditions
\begin{equation*}
\left[\frac{\p u^1}{\p\xi_2}\right]_\g=0,\quad [u^1]_\g=0.
\end{equation*}
For $|x_2-\frac{\pi}{2}|>2\e^{\frac{1}{2}}$ we thus have
\begin{equation*}
u^\e(x)=u^0(x)+\eta u^1(x)+\ldots
\end{equation*}
and therefore
\begin{equation*}
\|u^\e-u^0\|_{W_2^1(\Om^\e)}\geqslant C\eta.
\end{equation*}
Comparing this estimate with the right hand side of (\ref{2.13a}), we see that they differ just by $|\ln\eta|^{1/2}$ and in this sense estimate (\ref{2.13a}) is close to the sharp one.

If $a\equiv0$, in (\ref{8.3}) we let $h_0:=0$, $h_1^\pm:=1$. The external and internal layers are again introduced by (\ref{8.26}), but with $v^0\equiv0$. Function $v^1$ should solve problem (\ref{8.28}) with $a=0$ and behave at infinity as
$v^1(\xi,x_1)=\chi_2(x_1)\xi_2+\ldots$, $\xi_2\to\pm\infty$.
It reads as
$v^1(\xi,x_1)=\chi_2(x_1)Z_\eta^N(\xi)$,
where an auxiliary function $Z_\eta^N$ is described by the following lemma.

\begin{lemma}\label{lm8.3}
The problem
\begin{equation*}
\D_\xi Z_\eta^N=0\quad\text{in}\quad \Pi_\eta,\qquad \frac{\p Z_\eta^N}{\p\nu}=0\quad\text{on}\quad \p\tht_\eta,
\end{equation*}
with periodic conditions on the lateral boundaries has the unique solution behaving at infinity as
\begin{equation}
Z_\eta^N(\xi)=\xi_2 \pm \frac{\eta^2}{1-\frac{\pi^2\eta^2}{28}} + O(e^{-\frac{2\pi}{3}|\xi_2|}),\quad \xi_2\to\pm\infty.
\end{equation}
This solution can be represented as
\begin{equation*}
Z_\eta^N(\xi)=Z_0(\xi)+\frac{\eta^2}{1-\frac{\pi^2\eta^2}{28}} \frac{\p Z_0}{\p\xi_2}+\wt{Z}_\eta^N(\xi),
\end{equation*}
where function $\wt{Z}_\eta^N$ is $3$-periodic w.r.t. $\xi_1$, decays exponentially as $\xi_2\to\pm\infty$ and satisfies the uniform in $\eta$ estimate
\begin{equation*}
\|\wt{Z}_\eta^N\|_{W_2^1(\Pi_\eta)}\leqslant C\eta.
\end{equation*}
\end{lemma}

The asymptotics for $u^\e$ reads as
\begin{equation*}
u^\e(x)=u^0(x)\chi_2^\e(x_2) +  \e\chi_2(x_1) Z_\eta^N \left(\frac{x_1}{\e},\frac{x_2-\frac{\pi}{2}}{\e}\right) + \ldots
\end{equation*}
And by analogy with (\ref{8.16}), (\ref{8.18}) we get
\begin{equation*}
\|u^\e-u^0\|_{W_2^1(\Om^\e)} \geqslant C\e\eta^2.
\end{equation*}
This estimate differ from the right hand side of (\ref{2.13b}) just by $|\ln\eta|^{\frac{1}{2}}$  and hence, estimate (\ref{2.13b}) is close to be order sharp.

\bigskip

D.B. was partially supported by grant of RFBR, grant of President of Russia for young scientists-doctors of sciences (MD-183.2014.1) and Dynasty fellowship for young Russian mathematicians. The research presented in Section~8 was supported by Russian Science Foundation (project no. 14-11-00078).

\end{document}